\documentclass[10pt,paper,twistier,web]{ieeecolor}

\usepackage{generic}
\usepackage{cite}
\usepackage{amsmath,amssymb,amsfonts,latexsym,mathrsfs}
\usepackage{nccmath}
\usepackage{bm}
\usepackage{url}
\usepackage{subfig}
\usepackage{algorithmic}
\usepackage{graphicx}
\usepackage{textcomp}
\usepackage{comment}
\usepackage[switch]{lineno}

\graphicspath{{figures/}}

\newtheorem{definition}{Definition}
\newtheorem{theorem}{Theorem}
\newtheorem{corollary}{Corollary}
\newtheorem{example}{Example}
\newtheorem{lemma}{Lemma}
\newtheorem{remark}{Remark}

\newcommand{\ket}[1]{\left|#1\right\rangle}
\newcommand{\bra}[1]{\left\langle#1\right|}
\newcommand{\norm}[1]{\|#1\|}
\renewcommand{\vec}{\bm}
\newcommand{\diag}{\operatorname{diag}}
\newcommand{\rank}{\operatorname{rank}}
\newcommand{\Ad}{\operatorname{Ad}}
\newcommand{\Tr}{\operatorname{Tr}}
\renewcommand{\ss}{{\rm ss}}
\newcommand{\A}{\bm A}
\renewcommand{\L}{\bm L}
\renewcommand{\S}{\bm S}
\newcommand{\I}{\bm I}
\newcommand{\U}{\bm U}
\renewcommand{\H}{\mathbf{H}}
\newcommand{\V}{\mathcal{V}}
\newcommand{\W}{\mathcal{W}}
\newcommand{\bPhi}{\bm \Phi}
\newcommand{\T}{\bm T}
\begin{document}

\title{Robust Control Performance for Open Quantum Systems}
\author{S.~G.~Schirmer, \IEEEmembership{Member, IEEE},
F.~C.~Langbein, \IEEEmembership{Member, IEEE},
C.~A.~Weidner, \\
E.~Jonckheere, \IEEEmembership{Life Fellow, IEEE}
\thanks{This research was supported in part by NSF Grant IRES-1829078.}
\thanks{SGS is with the Faculty of Science \& Engineering, Swansea University, Swansea SA2 8PP, UK (e-mail: s.m.shermer@gmail.com).}
\thanks{FCL is with the School of Computer Science and Informatics, Cardiff University, Cardiff CF24 4AG, UK (e-mail: frank@langbein.org).}
\thanks{CAW is with the Quantum Engineering Technology Laboratories, University of Bristol, Bristol BS8 1FD, United Kingdom (e-mail: c.weidner@bristol.ac.uk).}
\thanks{EAJ is with the Department of Electrical and Computer Engineering, University of Southern California, Los Angeles, CA 90089 USA (e-mail: jonckhee@usc.edu).}
}
\maketitle

\begin{abstract}
Robust performance of control schemes for open quantum systems is investigated under classical uncertainties in the generators of the dynamics and nonclassical uncertainties due to decoherence and initial state preparation errors. A formalism is developed to measure performance based on the transmission of a dynamic perturbation or initial state preparation error to the quantum state error. This makes it possible to apply tools from classical robust control such as structured singular value analysis. A difficulty arising from the singularity of the closed-loop Bloch equations for the quantum state is overcome by introducing the \#-inversion lemma, a specialized version of the matrix inversion lemma. Under some conditions, this guarantees continuity of the structured singular value at $s = 0$. Additional difficulties occur when symmetry gives rise to multiple open-loop poles, which under symmetry-breaking unfold into single eigenvalues. The concepts are applied to systems subject to pure decoherence and a general dissipative system example of two qubits in a leaky cavity under laser driving fields and spontaneous emission. A nonclassical performance index, steady-state entanglement quantified by the concurrence, a nonlinear function of the system state, is introduced. Simulations confirm a conflict between entanglement, its log-sensitivity and stability margin under decoherence.
\end{abstract}


\section{Introduction}

\IEEEPARstart{Q}{uantum} control offers techniques to steer the dynamics of quantum systems. This is essential for enabling a wide range of applications for quantum technologies. However, uncertainties arising from limited knowledge of Hamiltonians, decoherence processes and initial state preparation errors impact the effectiveness of the control schemes. While classical robust control has developed effective solutions for such situations that apply relatively easily to quantum optics~\cite{Petersen2013}, they do not apply straightforwardly to other areas of quantum control such as spin systems. To consider the robustness of quantum control strategies in the presence of uncertainties, we develop a formalism where the performance is measured by the transmission $\T_{\vec{z},\vec{w}}(s,\delta)$ from the dynamic perturbation $\vec{w}$ (including state preparation errors) to the error $\vec{z}$ on the quantum state when such transmission is subject to structured uncertainties of strength $\delta$. It is tacitly assumed that this response has been made $H^\infty$-small by the control design under nominal values of the parameters in the Hamiltonian and decoherence. Robust performance is therefore defined as the ability of $\T_{\vec{z},\vec{w}}(\imath\omega,\delta)$ to remain within identifiable bounds for $\delta \ne 0$. Since uncertainties in the Hamiltonians and Lindbladians are often \emph{structured}, it is natural to quantify robustness of the performance using structured singular values. A generic difficulty that arises for quantum systems is that trace conservation of the density matrix $\rho$ imposes a closed-loop pole at $s=0$ in the $\T_{\vec{z},\vec{w}}(s,\delta)$ dynamics. This creates a singularity in the dynamics at low frequencies, $\omega \approx 0$, mandating some revision of the traditional machinery of structured singular values and a special \emph{matrix \#-inversion lemma}, similar to, but distinct from the matrix pseudo-inversion lemma~\cite{matrix_pseudo_inversion_lemma, operator_pseudo_inversion_lemma}. Other difficulties addressed by our formalism include multiple poles, either structurally stable like the pole at $s=0$ or removable by perturbation of physically meaningful parameters. We further demonstrate applicability of the various concepts to two cases that have no classical counterparts: pure dephasing acting in the Hamiltonian basis (i.e., an eigenbasis of the Hamiltonian) and dissipative cavity dynamics. A deeper underlying issue is whether the classical limitation of conflict between performance and sensitivity of performance to uncertainties holds in coherent quantum control and in the presence of decoherence.

After reviewing quantum dynamics in Sec.~\ref{sec:quantum_dynamics}, the general error dynamics with transfer matrix $\T_{\vec{z},\vec{w}}$ that should be robust against uncertainties in the Hamiltonian~\cite{Edmond_IEEE_AC} and decoherence~\cite{CDC_decoherence} is introduced in Sec.~\ref{sec:modelling}. Preparation error response requires a different formulation departing from classical robust performance as in~\cite{ssv_mu,Zhou}. In Sec.~\ref{sec:pure_dephasing}, the case of pure dephasing in an eigenbasis of the Hamiltonian is developed and analytic bounds for the error transmission $\T_{\vec{z},\vec{w}}$ are derived. Sec.~\ref{sec:general_dissipative} deals with generic dissipative quantum systems and develops a generalized framework to deal with the $s=0$ singularity. In Sec.~\ref{sec:2qubits}, robust performance for generic dissipative dynamics is illustrated by the case study of two qubits in a cavity. This simple example allows the formulation of another novelty in robust control: a nonlinear performance index in the form of the concurrence, a measure of entanglement. We note, however, that the analysis here can also be applied to linear measures such as fidelity.

\section{Quantum Dynamics and Uncertainties} \label{sec:quantum_dynamics}

\subsection{Schr\"odinger and Lindblad Equations}

The dynamics of a quantum system whose pure states $\ket{\Psi(t)}$ are wavefunctions in a Hilbert space $\H=\mathbb{C}^N$, are typically described by the Schr\"odinger equation, $\tfrac{d}{dt}\ket{\Psi(t)} = -\imath H \ket{\Psi(t)}$ (in a system of units where the reduced Planck constant $\hbar=1$), or the Liouville-von Neumann equation for density operators $\rho$,
\begin{equation}\label{e:Liouville}
  \tfrac{d}{dt} \rho(t) = -\imath [H,\rho(t)].
\end{equation}
Here, $H$ is the Hamiltonian of the system, $\rho$ is a (bounded) Hermitian operator on $\H$ with $\Tr(\rho)=1$ and $[A,B] = AB-BA$ is the usual matrix commutator. For pure states, the density operator is simply the projector onto $\ket{\Psi}$, i.e., $\rho = \ket{\Psi}\bra{\Psi}$. If $\dim\H=N<\infty$, $H$ and $\rho$ can be represented by $N\times N$ Hermitian matrices. The Liouville-von Neumann formulation can easily be extended to describe open system dynamics by adding Lindbladian terms to the right-hand side of Eq.~\eqref{e:Liouville},
\begin{equation}\label{e:QME2}
  \mathfrak{L}(V_k) \rho = V_k \rho V_k^\dag -\tfrac{1}{2} (V_k^\dag V_k \rho+\rho V_k^\dag V_k).
\end{equation}
This results in the Lindblad master equation
\begin{equation}~\label{e:QME1}
  \frac{d}{dt}\rho(t)=-\imath [H,\rho(t)] + \sum_{k}\gamma^2_{k} \mathfrak{L}(V_k)\rho(t),
\end{equation}
where the $V_k$'s are the jump operators~\cite{open_geometric_phase} and the $\gamma^2_{k}$'s are the decoherence rates, which can be interpreted as the \emph{strengths} of the \emph{structured} perturbations defined by $\mathfrak{L}(V_k)$.

The most natural formulation of quantum control is through a control-dependent Hamiltonian $H = H[\vec{f}]$, which, in its simplest form $H[\vec{f}] = H_0 + \sum_{m=1}^M f_m(t) H_m$, is linear in the controls $f_m(t)$, leading to a bilinear system. The system Hamiltonian is modified by $\vec{f}$, which can also alter the system-bath interaction, leading, in the simplest case, to a control-dependent Lindbladian
\begin{equation}\label{e:QME1contr}
  \frac{d}{dt}\rho(t)= -\imath [H[\vec{f}],\rho(t)] + \sum_{k}\gamma^2_{k} \mathfrak{L}(V_k[\vec{f}])\rho(t).
\end{equation}
For general system-bath interactions, the control $\vec{f}$ may result in a structurally different master equation~\cite{domenico_CDC,singular_vs_weak_coupling}.

In the remainder of the paper, we adopt Eq.~\eqref{e:QME1contr} with time-independent  control fields $f_m$.

\subsection{Initial State Preparation Errors versus Dephasing}

When studying quantum dynamics subject to structured perturbations, note that some perturbations may be indistinguishable, e.g., initial state preparation error and decoherence. If we measure the difference between an actual state $\tilde{\rho}$ and a desired state $\rho$ at a certain time, we cannot determine if the error is the result of dynamic dephasing of a perfectly prepared pure initial state, a mixed initial state evolving perfectly according to unitary dynamics, or a combination of both. This is illustrated with a simple example that also serves as a ``warm-up exercise.''

\begin{example}
Consider a two-level system with $\ket{0}=[1,0]^T$ and $\ket{1}=[0,1]^T$. Assume the system is prepared in a pure state $\ket{\psi_0}=(\ket{0}+\ket{1})/\sqrt{2}$ and evolves under the Hamiltonian $H=\omega\sigma_z$ with $\sigma_z =\diag(-1,1)$ while dephasing $V=\sigma_z$ acts in the Hamiltonian basis ($[H,V]=0$) at a rate $\delta=\gamma^2$. In this case, Eq.~\eqref{e:QME1} gives
\begin{equation*}
  \rho_0 = \rho(0) = \frac{1}{2}\begin{bmatrix} 1 & 1 \\ 1 & 1\end{bmatrix}, \quad
  \rho(t) = \frac{1}{2}\begin{bmatrix} e^{2\imath\omega t} & e^{-2\delta t} \\ e^{-2\delta t} & e^{-2\imath \omega t}\end{bmatrix}.
\end{equation*}
Alternatively, a mixed initial state $\tilde{\rho}_0$ evolving under the same Hamiltonian $H=\omega \sigma_z$ without dephasing results in
\begin{equation*}
  \tilde{\rho}_0 = \tilde{\rho}(0) = \frac{1}{2}\begin{bmatrix} 1 & e^{-\tau\delta} \\ e^{-\tau\delta} & 1\end{bmatrix}, \quad
  \tilde{\rho}(t) = \frac{1}{2}\begin{bmatrix} e^{2\imath\omega t} & e^{-\tau\delta} \\ e^{-\tau\delta} & e^{-2\imath\omega t}\end{bmatrix}.
\end{equation*}
Clearly, the two solutions are indistinguishable at $t=\tau/2$.
\end{example}

\subsection{Bloch Equation}\label{sec:Bloch}

To simplify the analysis we reformulate the dynamics Eq.~\eqref{e:QME1contr} as a linear ODE for a real state vector $\vec{r}$, derived from the classical technique~\cite{PhysRevA.93.063424,PhysRevA.81.062306,JPhysA37,neat_formula} of expanding $\rho$ and the Hamiltonian with respect to a suitable basis for the operators on the Hilbert space $\H$ such as the generalized Pauli matrices~\cite{Gell_Mann_matrices}. Let $\{\sigma_n\}_{n=1}^{N^2}$ be an orthonormal basis for the $N\times N$ Hermitian matrices with $\sigma_{N^2} = \tfrac{1}{\sqrt{N}} I$. Then defining $\vec{r}=(r_n)_{n=1}^{N^2}$ with $r_n = \Tr(\sigma_n\rho)$ and
\begin{subequations}\label{e:ASofHV}
\begin{alignat}{2}
  \A_H[\vec{f}]_{mn} &= \mathrm{Tr} && (\imath H[\vec{f}][\sigma_m,\sigma_n]),\\
  (\L_{k}[\vec{f}])_{mn} &= \mathrm{Tr}&&\mathopen{}\mathclose{\left(V_k^\dag[\vec{f}]\sigma_mV_k[\vec{f}]\sigma_n \right.}\\\nonumber
  & &&~\left. -\tfrac{1}{2}V^\dag_k[\vec{f}]V_k[\vec{f}]\{\sigma_m,\sigma_n\}\right),
\end{alignat}
\end{subequations}
leads to the state-space equation,
\begin{equation}\label{eq:Bloch1}
  \tfrac{d}{dt}{\vec{r}}(t) = \left(\A_H[\vec{f}]+\gamma^2 \A_L[\vec{f}]\right)\vec{r}(t),
\end{equation}
where $\gamma^2 \A_L[\vec{f}]$ is short for $\sum_{k}\gamma^2_{k} \L_{k}[\vec{f}]$. Observe that $\A_H$ is real and anti-symmetric while $\A_L$ is symmetric.

From $r_{N^2} = \Tr(N^{-1/2} \rho) = N^{-1/2}$, it follows that $\tfrac{d}{dt}{r}_{N^2}(t)\equiv 0$ and the last row of $\A_H+\gamma^2 \A_L$ vanishes for all $\gamma$. Thus, in general, $\rank(\A_H+\gamma^2 \A_L)\leq N^2-1$, $\forall \gamma$. Moreover, $\rank(\A_H)\leq N^2-1$ and $\rank(\A_L)\leq N^2-1$, indicating that this rank deficiency is a consequence of the choice of the basis operator $\sigma_{N^2}$ and independent of the dynamical generators.

If the linearity of Eq.~\eqref{eq:Bloch1} relative to $\vec{f}$ is preserved, the RHS of the master equation can be rewritten as
\begin{equation}\label{eq:Bloch0}
  \frac{d}{dt}{\vec{r}}(t) = \left( \A_H[0]+\gamma^2 \A_L[0] +\sum_{m=1}^M f_m(t) \A_m\right) \vec{r}(t),
\end{equation}
where the $\A_m$'s are obtained by adaptation of Eqs~\eqref{e:ASofHV}, and
\begin{equation} \label{eq:fb}
  \left[ \sum_{m=1}^M f_m(t) \A_m \right] \vec{r}(t) =: \vec{u}(\vec{r}(t),t)
\end{equation}
can be interpreted as state feedback. Since the $f_m$ are time-independent controls, the resulting system is a linear time-invariant (LTI) system with autonomous state feedback~\cite{Edmond_IEEE_AC}. Note, however, that controlling such systems by choosing some $f_m$ to produce specified eigenvalues for $\A_H[0]+\gamma^2 \A_L[0]+\sum_m f_m \A_m$ does not reduce to the well-known controllability pole placement method~\cite{bilinear_constant_input}.

\subsection{Genericity of the Eigenstructure}

Quantum systems often have multiple eigenvalues and degeneracies. It is useful to distinguish two cases. \textit{Structurally stable} degeneracies are those that cannot be eliminated by invoking parameter drift. An example of this is the state transition matrix $\bm A_H+\gamma^2 \bm A_L$ of the Bloch equation, which always has at least one zero eigenvalue due to trace conservation. An $N$-level system subject to pure dephasing in the Hamiltonian basis always has $N$ zero-eigenvalues, regardless of perturbations of the Hamiltonian or dephasing rates. Other degeneracies are \textit{structurally unstable} and can be lifted by perturbations. Consider a quantum system with Hilbert-space dimension $N$. Let $\{\Pi_k\}_{k=1}^{\bar{N}\leq N}$ be the family of projectors onto the respective eigenspaces of the Hamiltonian associated with the eigenvalues $\lambda_k(H)$ with $\sum_{k=1}^{\bar{N}}n_k=N$. For many systems subject to symmetries such as the Hamiltonians of chains and rings~\cite{chains_QINP,rings_QINP} under uniform $J$-coupling, the eigenvalues will have multiplicities $n_k=\rank(\Pi_k)>1$. Such multiplicities can be removed by modifying the $J$-couplings or other relevant parameters to break the symmetry. In general, three parameters are necessary to split a structurally unstable eigenvalue with multiplicity greater than one~\cite{von_neumann_wigner}. This statement is made rigorous in Appendix~\ref{a:multiple}. Conversely, under parameter drift, we cannot rule out eigenvalues of $H$ crossing at a particular time, which raises the question whether a continuous eigenbasis for their respective eigenspaces exists. This question was answered affirmatively by Dole\v{z}al's theorem and its generalization~\cite{Dolezal1,Dolezal2}, details of which can be found in Appendix~\ref{a:Dolezal}.

\section{Robust Performance in Open Systems} \label{sec:modelling}

Our goal is to quantify robust performance of open quantum systems in the presence of \emph{structured} perturbations. Decoherence can be treated as a structured perturbation for a closed quantum system. Many other uncertainties in the system or control Hamiltonian can also be treated as structured perturbations. For example, we can model uncertainty as $H_\theta$, where $H_\theta$ depends affinely on a parameter $\theta \in \Theta \subset \mathbb{R}$, that is, $H_\theta=H_{\theta_0}+(\theta-\theta_0)S=H_{\theta_0}+\delta S$, where $\theta_0$ is the nominal value. Such a decomposition maps, via Eq.~\eqref{e:ASofHV}, to the Bloch state transition matrix.

We can quantify closed-loop performance of a quantum system operating under uncertainties by a $\delta$-dependent disturbance-to-error transmission function $\T_{\vec{z},\vec{w}}(s,\delta)$. Consider an unperturbed and a perturbed version of Eq.~\eqref{eq:Bloch0} with state vectors $\vec{r}_u$ and $\vec{r}_p$, respectively, evolving according to
\begin{subequations}\label{eq:Bloch3}
  \begin{align}
    \tfrac{d}{dt} {\vec{r}}_u(t) &= \A \vec{r}_u(t),\\
    \tfrac{d}{dt} {\vec{r}}_p(t) &= (\A + \delta \S)\vec{r}_p(t),
  \end{align}
\end{subequations}
where $\A$ is the Bloch operator for the ideal unperturbed system and $\delta \S$ is a perturbation to the dynamics of structure $\S$ and magnitude $\delta$ in the Bloch representation. The error vector $\vec{z}(t) = \vec{r}_p(t)-\vec{r}_u(t)$ satisfies either dynamical representation
\begin{subequations}
  \begin{align}
    \tfrac{d}{dt}{\vec{z}}(t) &= (\A + \delta \S) \vec{z}(t) + \delta \S \vec{w}_u(t),\label{eq:A1} \\
    \tfrac{d}{dt}{\vec{z}}(t) &= \A \vec{z}(t) + \delta \S \vec{w}_p(t),\label{eq:A2}
  \end{align}
\end{subequations}
driven by $\vec{w}_u(t)=\vec{r}_u(t)$ and $\vec{w}_p(t) = \vec{r}_p(t)$, respectively. Projecting $\vec{z}$ on some desired subspace yields the error in \emph{fidelity} due to uncertainties. Here, we adopt the \emph{unperturbed} formulation of Eq.~\eqref{eq:A1}.
The perturbed formulation is available in~\cite{CDC2021_mu}.

Taking the Laplace transform of Eq.~\eqref{eq:A1} yields
\begin{equation}\label{eq:sys1}
  (s \I - \A-\delta \S) \vec{\hat{z}}(s) = \delta \S \vec{\hat{w}}_u(s) + \vec{z}(0).
\end{equation}
If $s\I-\A-\delta\S$ is invertible and there is no initial state preparation error, $\vec{z}(0)=\vec{0}$, then
\begin{equation}\label{e:clue}
  \vec{\hat{z}}(s) = \T_{\vec{z},\vec{w}_u}^u(s,\delta) \vec{\hat{w}}_u(s),
\end{equation}
where the transfer matrix $\T_{\vec{z},\vec{w}_u}^u$ is given by
\begin{align} \label{e:scaling1}
  \T_{\vec{z},\vec{w}_u}^u
  &:= (s \I-\A-\delta \S)^{-1}\delta \S \\\nonumber
  &= \left[(s \I-\A-\delta \S)^{-1}-(s \I-\A)^{-1}\right] \left[(s \I-\A)^{-1}\right]^{-1},
\end{align}
scaling the error relative to the known system.
If there are initial state preparation errors, i.e., if $\vec{z}(0) \neq 0$ in Eq.~\eqref{eq:A1}, but $s\I-\A-\delta \S$ is invertible then
\begin{align}\label{eq:x0m1}
  \vec{\hat{z}}(s)
  &= \T_{\vec{z},\vec{w}_u}^u(s,\delta) \vec{\hat{w}}_u(s) + \T_{\vec{z},\vec{z}_0}^u(s,\delta) \vec{z}(0),
\end{align}
if we set $\T_{\vec{z},\vec{z}_0}^u(s,\delta) = (s\I-\A-\delta\S)^{-1}$ as in~\cite{ssv_mu}. The transfer function quantifies how errors due to imperfect dynamics or initial state preparation (Eq.~\eqref{eq:x0m1}) scale as functions of frequency and what the system's critical frequencies are.

The formulation of Eq.~\eqref{e:clue} also enables structured singular value analysis~\cite{ssv_mu,Zhou}. If $\bPhi(s):=s\I-\A$ is invertible then
\begin{align}\label{eq:MIL}
  (\bPhi(s)-\delta \S)^{-1}\delta\S
  &= [\bPhi(s)(\I-\bPhi(s)^{-1}\delta\S)]^{-1} \delta\S \nonumber \\
  &=(\I-\bPhi^{-1}(s) \delta \S)^{-1}\bPhi^{-1}(s) \delta \S.
\end{align}
This reveals that the error response $\T^u_{\vec{z},\vec{w}_u}$ is obtained from
\begin{equation}\label{e:Gzwu}
  \begin{pmatrix}\vec{\hat{v}}\\\vec{\hat{z}}(s)\end{pmatrix}
  =\underbrace{\begin{pmatrix}\bPhi^{-1}(s)\S & \bPhi^{-1}(s)\S\\\I & \bm 0 \end{pmatrix}}_{G_{\vec{z},\vec{w}_u}(s)}
   \begin{pmatrix} \vec{\hat{\eta}} \\ \vec{\hat{w}}_u(s)\end{pmatrix},
\end{equation}
with feedback $\vec{\hat{\eta}}=(\delta \I) \vec{\hat{v}}$. Following standard robust control, $\|\T^u_{\vec{z},\vec{w}_u}(s,\delta)\|$ can be computed as $1/\min\{\|\Delta_f\|: \det(\I+\T^u_{\vec{z},\vec{w}_u}(s,\delta)\Delta_f)=0\}$ where $\vec{w}_u=\Delta_f \vec{z}$ is a \emph{fictitious} feedback wrapped around $G_{\vec{z},\vec{w}_u}$, and $\Delta_f$ is a fully populated complex matrix. Both feedbacks can be combined in a matrix $\bm \Delta=\begin{pmatrix}\delta \I & 0\\0 & \Delta_f\end{pmatrix}$, which defines the set $\mathcal{D}$ of block-diagonal matrices $\diag(\bm M_1,\bm M_2)$, where $\bm M_1$ and $\bm M_2$ are $N^2\times N^2$ matrices with $\bm M_1$ real scalar and $\bm M_2$ fully populated and complex. The structured singular value specific to $\mathcal{D}$,
\begin{equation*}
  \mu_{\mathcal{D}}(G_{\vec{z},\vec{w}_u}(s))
  =\frac{1}{\min \{\|\bm\Delta \in \mathcal{D}\| : \det(\I+G_{\vec{z},\vec{w}_u}(s) \bm \Delta)=0 \}}
\end{equation*}
is a measure of robust performance~\cite[Th. 10.8]{Zhou}:
\begin{theorem}\label{th:robust_performance}
If $\bPhi(s)$ is invertible then $\|\T^u_{\vec{z},\vec{w}}(s,\delta)\| \leq \mu_{\mathcal{D}}(G_{\vec{z},\vec{w}_u}(s))$ for all $\delta < [\mu_{\mathcal{D}}(G_{\vec{z},\vec{w}_u}(s))]^{-1}$.
\end{theorem}

Similarly, $\vec{\hat{z}}_{\vec{z}_0} = \T_{\vec{z},\vec{z}_0}^u(s,\delta)\vec{z}(0)$ can be obtained from
\begin{equation}\label{e:Gzz0}
  \begin{pmatrix} \vec{\hat{v}}_1 \\ \vec{\hat{z}}_{\vec{z}_0}(s) \end{pmatrix}=
  \underbrace{\begin{pmatrix}\S\bPhi^{-1}(s) & \S\bPhi^{-1}(s)\\
  \bPhi^{-1}(s) & \bPhi^{-1}(s) \end{pmatrix}}_{G_{\vec{z},\vec{z}_0}(s)}
  \begin{pmatrix} \vec{\hat{\eta}} \\ \vec{z}_0\end{pmatrix},
\end{equation}
with feedback $\vec{\hat{\eta}}=(\delta \I) \vec{\hat{v}}_1$ if $\bPhi(s)$ is invertible. Moreover, $\|\T_{\vec{z},\vec{z}_0}^u(s,\delta)\|$ is simultaneously computed via the compound feedback $\bm \Delta$ as defined above, and the related robust performance theorem is a straightforward adaptation of Th.~\ref{th:robust_performance}.

However, these results are \emph{not} directly useful for quantum systems because the Bloch matrix $\A$ of an open quantum system that enforces constancy of the trace always has an eigenvalue at $0$, so that $\bPhi^{-1}(s)$ has a pole at $s=0$, invalidating Th.~\ref{th:robust_performance}. The heuristic approach $s \approx 0$ has been  attempted~\cite{CDC_decoherence,ssv_mu}, but proper treatment requires a modification of the matrix inversion lemma. Some matrix pseudo-inversion lemmas~\cite{matrix_pseudo_inversion_lemma, operator_pseudo_inversion_lemma} have been proposed but do not apply in this context because their application is restricted to symmetric matrices. To address this issue, a specialized \emph{matrix \#-inversion lemma} is proposed in Secs.~\ref{sec:pseudo_inversion} and~\ref{sec:sing_prep_error}.

\section{Pure Dephasing in the Hamiltonian Basis} \label{sec:pure_dephasing}

In this section we study the performance of controlled quantum systems subject to dephasing, a typically undesired behavior commonly encountered for quantum systems interacting weakly with an environment. In this special case we can assume that dephasing acts in the eigenbasis of the Hamiltonian, i.e., $[H_\theta, V_\theta] =0$, where $\theta$ is a real uncertain parameter affecting the Hamiltonian and hence the decoherence. We treat dephasing as a perturbation to the closed system dynamics, emphasizing the decomposition $\A_\theta + \delta \S_\theta$, where $\A_\theta$ is the Bloch matrix corresponding to $H_\theta$ and $\delta\S_\theta$ is the Bloch matrix corresponding to $\gamma^2\mathfrak{L}(V_\theta)$ where $\delta=\gamma^2$.

\subsection{Pure Dephasing in the Hamiltonian Formulation}

By Dole\v{z}al's theorem in Appendix~\ref{a:Dolezal}, eigenvectors $\{\vec{u}_k(\theta)\}_{k=1}^N$ of $H_\theta$ can be chosen to form an orthonormal basis such that the unitary operator $U_\theta = (\vec{u}_1(\theta) \cdots \vec{u}_N(\theta))$ depends continuously on $\theta$. From $U_\theta$, in accordance with Lemma~\ref{l:PQQP}, we construct
\begin{equation*}
  V_\theta = U_\theta\diag(\lambda_1(V_\theta),\dotsc,\lambda_N(V_\theta))U_\theta^\dagger,
\end{equation*}
which secures $[H_\theta,V_\theta]=0$. The eigenvalues $\lambda_k(V_\theta)$ can be chosen arbitrarily provided they are real, positive, and their multiplicities are consistent with those of $H_\theta$.

From the eigenvectors $\vec{u}_k(\theta)$, one can construct a set of projectors $\{\Pi_k(H_\theta)\}_{k=1}^{\bar{N}\leq N}$ onto the (orthogonal) simultaneous eigenspaces of $H_\theta$ and $V_\theta$ such that $\sum_{k=1}^{\bar{N}}\Pi_k(H_\theta)=I_{\mathbb{C}^{N}}$ is a resolution of the identity on the full Hilbert space $\H$ and
\begin{equation*}
  H_\theta = \sum_{k=1}^{\bar{N}} \lambda_k(H_\theta) \Pi_k(H_\theta), \quad
  V_\theta = \sum_{k=1}^{\bar{N}} \lambda_k(V_\theta) \Pi_k(H_\theta).
\end{equation*}
$\lambda_k(H_\theta)$ and $\lambda_k(V_\theta)$ are the respective real eigenvalues of $H_\theta$ and $V_\theta$. $\bar{N}\leq N$ is the number of \emph{distinct} eigenvalues of $H_\theta$. To simplify the notation we now drop the $\theta$-dependency.

Pre-/post-multiplying the master Eq.~\eqref{e:QME1} with Lindblad term~\eqref{e:QME2} by $\Pi_k(H)$ and $\Pi_\ell(H)$, respectively, yields
\begin{equation}\label{e:rho_k_ell}
  \Pi_k(H)\dot{\rho}(t) \Pi_{\ell}(H)= (-\imath \omega_{k\ell}+\delta\gamma_{k\ell}) \Pi_k(H) \rho(t) \Pi_\ell(H),
\end{equation}
with ${\omega}_{k\ell}= \lambda_k(H)-\lambda_\ell(H)$ and ${\gamma}_{k\ell}=-\tfrac{1}{2}(\lambda_k(V)-\lambda_\ell(V))^2\le 0$. The solution to this equation is
\begin{equation*}
  \Pi_k(H)\rho(t)\Pi_\ell(H) = e^{-t(\imath {\omega}_{k\ell}-\delta{\gamma}_{k\ell})} \Pi_k(H) \rho_0 \Pi_\ell(H).
\end{equation*}
Since $\sum_{k=1}^{\bar{N}} \Pi_k(H) = I$, the full solution is found as $\rho(t)=\sum_{k,\ell=1}^{\bar{N}}\Pi_k(H)\rho(t)\Pi_\ell(H)$, which yields
\begin{equation}\label{e:varrho_of_t_solution}
  \rho(t) = \sum_{k,\ell=1}^{\bar{N}} e^{-t(\imath {\omega}_{k\ell}-\delta{\gamma}_{k\ell})} \Pi_k(H) \rho_0 \Pi_\ell(H).
\end{equation}
Moreover, from the above it is easily verified that $\Pi_k\rho(t)\Pi_k=\Pi_k\rho_0\Pi_k$ for $k=1,\dotsc,\bar{N}$. Therefore, as $\sum_{k=1}^{\bar{N}}\Pi_k(H) =I_{\mathbb{C}^{N}}$, the solution $\rho(t)$ has $\bar{N}$ invariant subspaces.

\subsection{Pure Dephasing in the Bloch Representation}\label{s:pure_dephasing}


\begin{lemma}\label{l:orthog}
Let $P$ and $Q$ be $N \times N$ Hermitian operators with $\Tr(P^\dagger Q)=0$. Then their Bloch vectors $\vec{p}=(\Tr(P^\dagger \sigma_n))_{n=1}^{N^2}$, $\vec{q}=(\Tr(Q^\dagger \sigma_n))_{n=1}^{N^2}$ are orthogonal, $\vec{p}^T\vec{q}=0$.
\end{lemma}
\begin{proof}
Expand $P$ and $Q$ in terms of the basis $\{\sigma_n\}_{n=1}^{N^2}$ of the set of Hermitian $N \times N$ operators.
\end{proof}

\begin{theorem}\label{t:HVAS}
Let $\bar{N}=N$. If $[H,V]=0$ in the quantum master Eq.~\eqref{e:QME1}, then $[\A,\S]=0$ in the Bloch equation, and the kernels of $\A$ and $\S$ coincide and are both $N$-dimensional.
\end{theorem}
\begin{proof}
See Appendix~\ref{a:DephasingHamiltonianEigenbasis}.
\end{proof}

\begin{figure*}
\null\hfill
\subfloat[$\|\bPhi(\imath\omega)^\#\|$ vs $\|\bPhi^+(\imath\omega)\|$] {\includegraphics[width=0.3\textwidth]{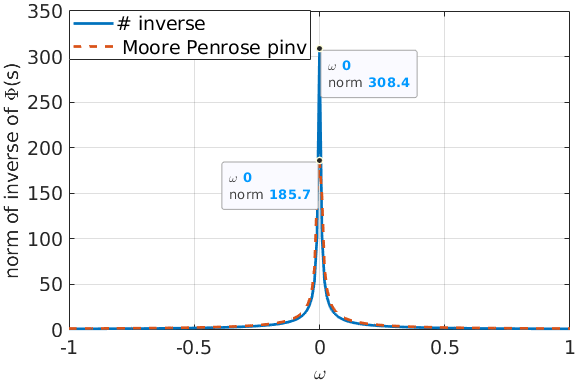}}\hfill
\subfloat[$\|T_{\vec{z},\vec{w}_u}^{u}(\imath\omega,\delta\S_k)\|$ for $\delta=0.1$] {\includegraphics[width=0.3\textwidth]{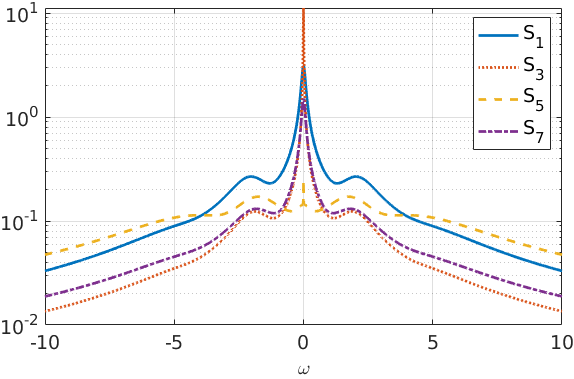}}\hfill
\subfloat[$\|T_{\vec{z},\vec{w}_u}^{u}(\imath\omega,\delta\S_k)\|$ for $\delta=1$] {\includegraphics[width=0.3\textwidth]{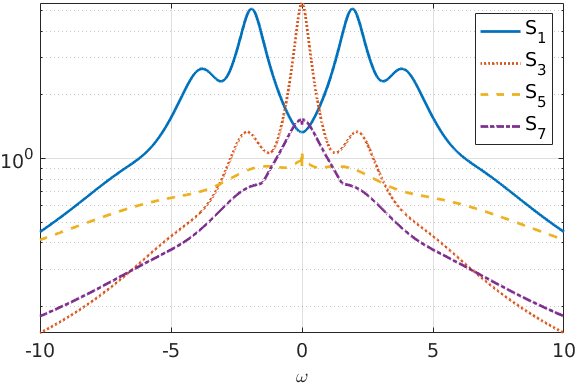}}\hfill\null
\caption{(a) Comparison of the norm of the inverse of $\bPhi(s)$ for the \#-inverse and the Moore-Penrose pseudo-inverse ($(\cdot)^+$).
(b,c) Error gains $\|T_{\vec{z},\vec{w}_u}^{u}(\imath\omega,\delta\S_k)\|$ as a function of frequency for the structured uncertainties in Eq.~\eqref{eq:S1234567} and different sizes $\delta\in\{0.1,1\}$. Due to the pairwise similarities $\S_1 \sim \S_2$, $\S_3 \sim \S_4$ and $\S_6 \sim \S_7$, the cases $\S_2$, $\S_4,$ and $\S_6$ are not plotted.}
\label{fig:2qubit_example_gain}
\end{figure*}

\subsubsection{Simultaneous diagonalization approach}

We first consider the case where $H$ does not depend on $\theta$. Recalling $[H,V]=0$, then by Th.~\ref{t:HVAS}, $\A$ and $\S$ are $N^2\times N^2$ matrices of rank $\le N^2-N$ with equality in the generic case. More specifically, by Lemma~\ref{l:PQQP}, $\A$ and $\S$ are simultaneously diagonalizable by a complex unitary matrix $\U$,
\begin{subequations}
\begin{align}
  \U^\dag \A \U & = \diag(\bm \Omega,0),\\
  \U^\dag \S \U & = \diag(\bm \Gamma,0),
\end{align}
\end{subequations}
where $\bm\Omega$ and $\bm\Gamma$ are diagonal matrices of rank $N^2-N$ in the generic case, with purely imaginary diagonal entries $\imath \omega_{k \ne \ell} = \imath(\lambda_k(H)-\lambda_\ell(H))$ for $\bm\Omega$ and purely real, negative diagonal entries $\gamma_{k \ne\ell}=-\tfrac{1}{2}(\lambda_k(V)-\lambda_\ell(V))^2$ for $\bm\Gamma$ (see~\cite{CDC_decoherence}). This allows us to rewrite Eqs.~\eqref{eq:A1} as
\begin{equation}
  \U^\dag \dot{\vec{z}} = \U^\dag (\A + \delta \S) \vec{z} + \delta \U^\dag \S \vec{w}_u(t).
  \end{equation}
Setting $\vec{\zeta} = \U^\dag \vec{z}$, $\vec{\upsilon}_u = \U^\dag \vec{w}_u$ we obtain
\begin{equation}\label{e:inlieu}
    \tfrac{d}{dt}\vec{\zeta} = \diag(\bm \Omega+\delta\bm \Gamma,0) \vec{\zeta} + \delta \diag(\bm \Gamma,0) \vec{\upsilon}_u.
\end{equation}
Note that, despite the \emph{real} form of the Bloch equations, $\vec{\zeta}$ and $\vec{\upsilon}_u$ are complex, as $\U^\dag$ is in general a complex unitary operator, although we could easily define an equivalent real form. Finally, we can partition the vectors $\vec{\zeta}$ and $\vec{\upsilon}_u$ such that
\begin{align*}
  \begin{pmatrix} \bm\Omega & 0 \\ 0 & 0 \end{pmatrix}
  \begin{pmatrix} \vec{\zeta}_1 \\ \vec{\zeta}_2 \end{pmatrix} =
  \begin{pmatrix} \bm\Omega \vec{\zeta}_1 \\ 0 \end{pmatrix}, \;
  \begin{pmatrix} \bm\Gamma & 0 \\ 0 & 0 \end{pmatrix}
  \begin{pmatrix} \vec{\upsilon}_{u,1} \\ \vec{\upsilon}_{u,2} \end{pmatrix}=
  \begin{pmatrix} \bm\Gamma \vec{\upsilon}_{u,1} \\ 0 \end{pmatrix}.
\end{align*}
We clearly have $\dot{\vec{\zeta}}_2=0$, i.e., $\vec{\zeta}_2(t)$ is constant. Therefore, the dynamics of the system are completely determined by $\vec{\zeta}_1(0)$ and the reduced model Bloch equation
\begin{equation}
  \dot{\vec{\zeta}}_1 = (\bm\Omega + \delta\bm\Gamma) \vec{\zeta}_1 + \delta\bm\Gamma \vec{\upsilon}_{u,1},
\end{equation}
a particular manifestation of Eq.~\eqref{eq:A1}. Generally, $\bm \Omega$ and $\bm \Omega+\delta\bm \Gamma$ are invertible, and taking the Laplace transform yields
\begin{equation}
  \widehat{\vec{\zeta}}_1(s) = (s \I - \bm\Omega - \delta \bm \Gamma)^{-1} \delta \bm \Gamma \vec{\hat{\upsilon}}_{u,1}(s).
\end{equation}
Therefore, $\vec{\hat{\zeta}}_1 = \T_{\vec{\zeta}_1, \vec{\upsilon}_{u,1}}^u(s,\delta) \vec{\hat{\upsilon}}_{u,1}$, and the transfer matrix from the disturbance input $\vec{\hat{\upsilon}}_{u,1}(s)$ to the error state $\vec{\hat{\zeta}}_1(s)$ is
\begin{equation}\label{eq:Tzw}
  \T_{\vec{\zeta}_1,\vec{{\upsilon}}_{u,1}}^u(s,\delta)
  = (s \I - \bm \Omega-\delta \bm \Gamma)^{-1} \delta \bm\Gamma.
\end{equation}
Taking $\bm\Omega = \diag(\imath\omega_{k\ne\ell})$ and $\bm\Gamma = \diag(\gamma_{k\ne\ell})<0$ we obtain
\begin{align*}
  \T_{\vec{\zeta}_1,\vec{\upsilon}_{u,1}}^u(\imath\omega,\delta)
  &= \diag( (\imath \omega - \imath \omega_{k\ell} - \delta \gamma_{k\ell} )^{-1}) \diag(\delta\gamma_{k\ell}) \\
  &= \diag\mathopen{}\mathclose{\left( \frac{\delta \gamma_{k\ell}}{\imath \omega - \imath \omega_{k\ell} - \delta \gamma_{k\ell}}\right)}.
\end{align*}
Taking the norm to be the largest singular value yields
\begin{equation}\label{eq:bound}
  \norm{\T_{\vec{\zeta}_1,\vec{\upsilon}_{u,1}}^u(\imath\omega,\delta)}_\infty
  = \max_{\omega,(k,\ell)} \left| \frac{\delta \gamma_{k\ell}}{\imath (\omega - \omega_{k\ell}) - \delta \gamma_{k\ell}} \right| =1,
\end{equation}
where the bound is obtained for $\omega = \omega_{k,\ell}$, i.e., if $\omega$ is an eigenfrequency of the system, for all $\delta$, including $\delta\to 0$.

\subsubsection{Simultaneous diagonalization under errors in the Hamiltonian}

Let $\omega_{k\ell}(\theta)=\lambda_k(H_\theta)-\lambda_\ell(H_\theta)$ and $\gamma_{k\ell}(\theta)=-\tfrac{1}{2}(\lambda_k(V_\theta)-\lambda_\ell(V_\theta))^2$. Assume the $k \ne \ell$ eigenvalues $-(\imath\omega_{k\ell}(\theta) +\delta \gamma_{k\ell}(\theta))$ of $H_\theta+\delta S_\theta$ have constant multiplicities, and let $\A_\theta$ and $\S_\theta$ be the Bloch representations of $H_\theta$ and $\mathfrak{L}(V_\theta)$, resp. By the proof of Th.~\ref{t:HVAS}, the corresponding eigenvalues of $\A_\theta+\delta \S_\theta$ do not cross, so we can simultaneously diagonalize $\A_\theta$ and $\S_\theta$. Hence, as $\U_\theta$ is unitary and, under the no-crossing hypothesis, depends continuously on $\theta$, we get, for all $\theta$,
\begin{equation}
  \U_\theta^\dag (\A_\theta +\delta \S_\theta) \U_\theta = \diag (\bm\Omega_\theta + \delta \bm\Gamma_\theta,0),
\end{equation}
where $\bm\Omega_\theta$, $\bm\Gamma_\theta$ display the perturbed eigenfrequencies, dampings, resp., on their diagonals. Proceeding as before, we set
\begin{equation}
  \vec{\zeta}(\theta) = \U_\theta^\dag \vec{z}(\theta), \quad \vec{\upsilon}(\theta) = \U^\dag_\theta \vec{w}(\theta),
\end{equation}
to obtain, in the unperturbed case of Eq.~\eqref{eq:A1},
\begin{align}\label{e:inlieutheta}
  {\tfrac{d}{dt}\vec{\zeta}(\theta)}
  &= \U^\dag_\theta (\A_\theta+\delta \S_\theta) \U_\theta \U_\theta^\dag \vec{z}(\theta) + \delta \U_\theta^\dag \S_\theta \U_\theta \U_\theta^\dag \vec{w}_u(\theta)\nonumber\\
  & = \diag(\bm\Omega_\theta + \delta \bm\Gamma_\theta,0) \vec{\zeta}(\theta) + \diag(\delta \bm\Gamma_\theta,0) \vec{\upsilon}_u(\theta),
\end{align}
under the assumption that the $\theta$-variation is slower than the dynamics as it is common practice in robust control. Eq.~\eqref{e:inlieutheta} is now used in lieu of Eq.~\eqref{e:inlieu}. If eigenvalues cross, under analyticity conditions, we can still proceed with block-diagonalization invoking the generalization~\cite{Dolezal2} of Dole\v{z}al's theorem~\cite{Dolezal1} (see Appendix~\ref{a:Dolezal}).

\begin{figure*} \centering
\subfloat[Maximum gain $\|\T_{\vec{z},\vec{w}_u}^{u}(\imath \omega,\delta \S) \|_\infty$]{\includegraphics[width=0.45\textwidth]{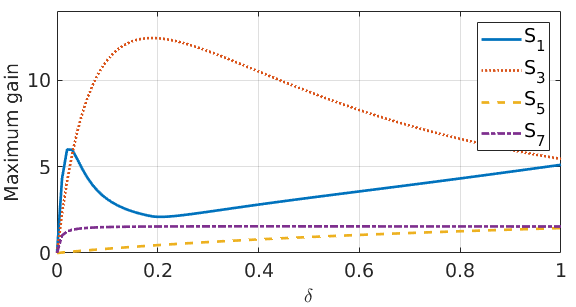}\qquad}
\subfloat[Frequency where maximum gain $\|\T_{\vec{z},\vec{w}_u}^{u}(\imath \omega,\delta \S)\|$ is achieved]{\includegraphics[width=0.45\textwidth]{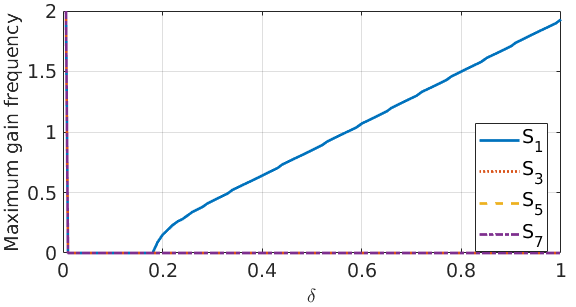}}
\caption{The maximum gain for the structured uncertainties in Eq.~\eqref{eq:S1234567} suggests that for small $\delta$ ($\delta<0.1$) the system is most sensitive to perturbations of type $\S_1$ while for larger $\delta$ sensitivity to $\S_3$ dominates.} \label{fig:2qubit_example_gain2}
\end{figure*}

\subsubsection{Robust performance approach}

Finally, we show that the above method is consistent with, and a simplification of, the robust performance approach of Sec.~\ref{sec:modelling} culminating in Th.~\ref{th:robust_performance},
but here restricted to $s\ne 0$. Setting $\bPhi(s) = s\I -\bm\Omega$, Eq.~\eqref{e:Gzwu} becomes
\begin{equation*}
  G_{\vec{\zeta}_1,\vec{\upsilon}_{u,1}} = \begin{pmatrix}
      \bPhi^{-1}(s) \bm\Gamma & \bPhi^{-1}(s) \bm\Gamma \\ \I & 0
    \end{pmatrix}, \quad {\bm \Delta} = \diag(\delta \I, \Delta_f).
\end{equation*}
It follows that
\begin{equation*}
  \det(\I + G_{\vec{\zeta}_1,\vec{\upsilon}_{u,1}} \bm \Delta)
  = \det [ \I + \delta \bPhi^{-1}(s) \bm \Gamma - \delta \bPhi^{-1}(s) \bm \Gamma \Delta_f].
\end{equation*}
Nominally from Th.~\ref{th:robust_performance}, $\Delta_f$ should be complex and fully populated to capture $\|\T^u_{\vec{\zeta}_1,\vec{\upsilon}_{u,1}}\|$. However, in the open quantum system situation, the following holds:
\begin{lemma}
If the fictitious feedback $\Delta_f$ is diagonal, $\diag\{\Delta_{f_{ii}}\}$, the optimal $\Delta_f$ yields $\|\T^u_{\vec{\zeta}_1,\vec{\upsilon}_{u,1}}\|=1/\|\Delta_f\|$.
\end{lemma}
\begin{proof}
We prove that a diagonal $\Delta_f$ captures $\|\T^u_{\vec{\zeta}_1, \vec{\upsilon}_{u,1}}\|$. To simplify the notation, the super/subscripts of $\T^u_{\vec{\zeta}_1, \vec{\upsilon}_{u,1}}$ are dropped and the subscript $f$ of $\Delta_f$ is dropped. Consider
\begin{equation}\label{e:optimization}
  1/\min_{\Delta}\{\|\diag\{\Delta_{{ii}}\}\|: \det(I+\diag\{T_{ii}\}\diag\{\Delta_{{ii}}\})=0\}.
\end{equation}
Assume $\min_{\Delta}$ is achieved for $\Delta^*$ and that $\|\diag\{\Delta_{{ii}}^* \}\|$ is achieved for $|\Delta^*_{{i_*i_*}}|$, where $i_*i_*$ might contain many indices.  (It is always possible to take $|\Delta^*_{i_*i_*}|$ strictly dominating all other $|\Delta^*_{ii}|$ unless all $|T_{ii}|$'s are equal in which case the lemma is trivial.) It is claimed that $1+T_{i_*i_*}{\Delta}_{{i_*i_*}}^*=0$. It suffices to show that $1+T_{jj}{\Delta}_{jj}^* \ne 0$, $\forall j \not\in i_*$. Assume by contradiction that $1+T_{jj}{\Delta}_{jj}^*=0$ for some $j \not\in i_*$. Then
\[
  \det(\I+\diag\{T_{ii}\}\diag\{\Delta_{jj}^*\ldots \Delta_{jj}^*\})=0
\]
while $\|\diag\{{\Delta}_{jj}^*\ldots {\Delta}_{jj}^*\}\| =|{\Delta}_{jj}^*|<|{\Delta}_{i_*i_*}^*|$, which is a contradiction to $\|\Delta^*\|$ achieved for $|\Delta^*_{i_*i_*}|$. From $1+T_{i_*i_*}\Delta_{i_*i_*}^*=0$ it follows that $\|\T\|=|T_{i_*i_*}|=1/|\Delta^*_{i_*i_*}|$, which equals~\eqref{e:optimization}.
\end{proof}

With this diagonal structure, $\det(\I+G \bm \Delta)$ vanishes if
\begin{equation*}
  1 + \delta \frac{\gamma_{k\ell}}{s-\imath\omega_{k\ell}} - \delta \frac{\gamma_{k\ell}}{s-\imath \omega_{k\ell}} (\Delta_f)_{k\ell,k\ell} = 0,
\end{equation*}
for some $k\ell$ indexing the diagonal. The above, solved for $(\Delta_f)_{k\ell,k\ell}$, gives
\begin{equation*}
  (\Delta_f)_{k\ell,k\ell} = \frac{s-\imath\omega_{k\ell}}{\delta \gamma_{k\ell}} + 1.
\end{equation*}
This assumes its minimum of $1$ for $s=\imath\omega_{k\ell}$ for all $\delta >0$. Choosing $\delta \leq 1$ yields $\min\norm{\bm \Delta}=1$. Thus, $\mu_\mathcal{D}(G(\imath \omega \ne 0)) = 1$, which is consistent with
Eq.~\eqref{eq:bound}.

\section{General Dissipative Dynamics}\label{sec:general_dissipative}

We apply and extend the formalism of Sec.~\ref{sec:modelling} to the general case of dissipative systems. Contrary to Sec.~\ref{sec:pure_dephasing}, the decoherence no longer acts in the Hamiltonian basis, that is, $[H,V] \ne 0$, and the uncertainty manifests itself as decoherence and/or uncertainties in the parameters of the Hamiltonian.
The corresponding Bloch operator $\A$ comprises \emph{both} the nominal Hamiltonian and the nominal decoherence dynamics, and all uncertainties are relegated to $\delta \S$. One obstacle to applying robust performance results to this case is the rank-deficiency of the $N^2\times N^2$ Bloch matrix $\A+\delta \S$, which generically has rank $N^2-1$. One of the main objectives of this section is to address this difficulty.

Despite its inconvenience, the rank deficiency of $\A+\delta \S$ can be exploited. It is common in physics to define a reduced $(N^2-1)\times (N^2-1)$ Bloch matrix $\overline{\A} +\delta \overline{\S}$ of full rank. This leads to an inhomogeneous Bloch equation for the reduced Bloch vector $\vec{s}$, where the trace component of $\vec{r}$ in Eq.~\eqref{eq:Bloch1} has been removed:
\begin{equation}
  \tfrac{d}{dt} \vec{s}(t) = (\overline{\A} +\delta \overline{\S}) \vec{s}(t) + \vec{c}.
\end{equation}
This equation is useful in some regards. If $\overline{\A}+\delta \overline{\S}$ is invertible (generic case) then the system has a unique steady-state $\vec{s}_{\rm ss} = -(\overline{\A} +\delta \overline{\S})^{-1} \vec{c}$, which can be shown to be \emph{globally} asymptotically stable~\cite{PhysRevA.81.062306}. Therefore, the steady-state is independent of the initial state and robust to initial state preparation errors. However, the control is sensitive to uncertainty in the Hamiltonian and dissipative processes.

\begin{figure*}\centering
\subfloat[Frequency sweep] {\includegraphics[width=.33\textwidth]{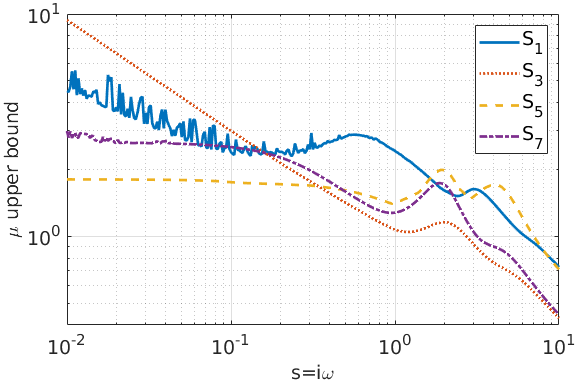}
\label{fig:S1357_upper_omega_details}}
\subfloat[$\Re(s)$ sweep]{\includegraphics[width=.33\textwidth]{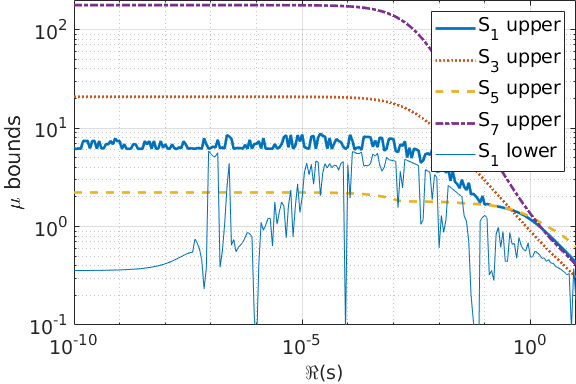}
\label{fig:S1357}}
\subfloat[$\Re(s)$ sweep for initial state error] {\includegraphics[width=0.33\textwidth]{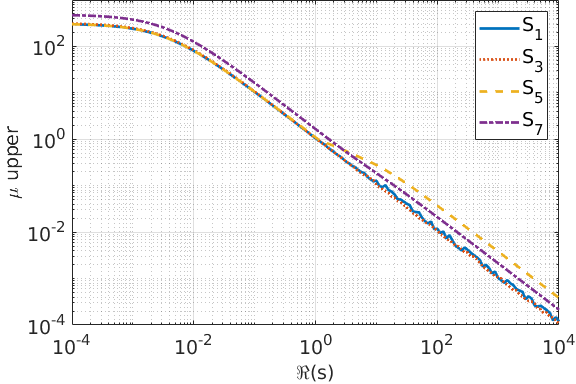}
\label{fig:S1357x0}}
\caption{(a) Upper bounds on the $\mu_\mathcal{D}$ bounding the error transmission
$\T_{\vec{z},\vec{w}_u}^u(\imath \omega, \delta \S_k)$ for the structured uncertainties in Eq.~\eqref{eq:S1234567} under frequency sweep $s=\imath\omega$. (b) upper and lower bounds for $s \downarrow 0$ along the real axis. Upper and lower bounds in (b) coincide for $\S_3$, $\S_5$, and $\S_7$. $\S_1$ displays aberrant behavior under frequency and real axis sweep. (c) Upper bounds on ${\mu}_\mathcal{D}^{(\#)}$ of initial state error transmission $\bar{T}_{\vec{z},\vec{z}_0}^u(s,\delta \S_k)$ for the structured uncertainties in Eq.~\eqref{eq:S1234567} for $s\downarrow0$ along the real axis.\label{fig:S1357_all}}
\end{figure*}

\subsection{Dynamic Disturbance Transmission}\label{sec:pseudo_inversion}

Consider Eq.~\eqref{eq:sys1} in the generalized uncertainty structure, but with no initial state preparation error. The singularity of $\bm \bPhi(s):=sI-\A$ at $s=0$ raised in Sec.~\ref{sec:modelling} leads to the question of solvability of the equation
\begin{equation}\label{eq:solution}
  (\bPhi(s)-\delta \S)\vec{\hat{z}}(s) = \delta \S\vec{\hat{w}_u}(s)
\end{equation}
for $\vec{\hat{z}}(s)$ when invertibility of $\bPhi(s)-\delta \S$ is not guaranteed. From the Bloch representation we know that the state vector is partitioned as $\vec{r} = [\vec{r}_1^T,c]^T$ where $\vec{r}_1$ is a column vector of length $N^2-1$ and $c$ is a constant, the value of which depends on the choice of basis, $c=\Tr[\bm\sigma_{N^2} \rho]$. Since both $\vec{r}_u$ and $\vec{r}_p$ have this structure, it follows that $\vec{z}=[\vec{z}_1^T;0]^T$, where $\vec{z}_1$ is a column vector of length $N^2-1$. Similarly, the Bloch matrix $\A$ and its structured perturbation $\S$ are of the form
\begin{equation} \label{eq:AS}
  \A = \begin{bmatrix} \A_{11} & \A_{12} \\ 0 & 0 \end{bmatrix}, \quad
  \S = \begin{bmatrix} \S_{11} & \S_{12} \\ 0 & 0 \end{bmatrix},
\end{equation}
where $\A_{11}$ and $\S_{11}$ are square matrices of size $N^2-1$, and
\begin{equation*}
  \bPhi(s) = s \I -\A = \begin{bmatrix} \bPhi_{11}(s) & \bPhi_{12} \\ 0 & s \end{bmatrix},
\end{equation*}
with $\bPhi_{11}(s) = s \I'-\A_{11}$, $\bPhi_{12} = -\A_{12}$, and $\I$ and $\I'$ the identity matrix in dimension $N^2$ and $N^2-1$, respectively. Thus Eq.~\eqref{eq:solution} can be written as
\begin{equation*}
  \begin{bmatrix}
  \bPhi_{11}(s)-\delta \S_{11} & \bPhi_{12}-\delta \S_{12} \\ 0 & s
  \end{bmatrix}
  \begin{bmatrix}
  \vec{\hat{z}}_1 \\ 0
  \end{bmatrix}
  = \delta \begin{bmatrix} \S_{11} & \S_{12} \\ 0 & 0 \end{bmatrix}
    \begin{bmatrix} \vec{\hat{w}}_1 \\ \hat{c} \end{bmatrix},
\end{equation*}
where $\hat{c}=c/s$. The above yields the equation for $\vec{\hat{z}}_1$,
\begin{equation} \label{eq:z1}
  (\bPhi_{11}(s)-\delta \S_{11}) \vec{\hat{z}}_1 = \delta (\S_{11} \vec{\hat{w}}_1 + \hat{c} \S_{12}).
\end{equation}
This shows that if $\bPhi_{11}(s)-\delta \S_{11}$ is invertible then there is a unique solution
$\vec{\hat{z}}_1 = (\bPhi_{11}(s)-\delta \S_{11})^{-1} \delta (\S_{11} \vec{\hat{w}}_1 + \hat{c} \S_{12})$, which converges to an asymptotically stable steady-state. If $\bPhi_{11}(s)-\delta \S_{11}$ is not invertible then there is a manifold of steady-states and stability is not guaranteed.

\begin{definition}
For $\A$ and $\S$ of the form of Eq.~\eqref{eq:AS}
with $\bPhi_{11}(s)-\delta \S_{11}$ invertible, we define the {\it \#-inverse} as
\begin{equation*}
  (\bPhi(s)-\delta \S)^\#
  = \begin{bmatrix}
  (\bPhi_{11}(s)-\delta \S_{11})^{-1} & 0 \\ 0 & 0 \end{bmatrix}.
\end{equation*}
\end{definition}
We can verify that this definition ensures that $\vec{\hat{z}}$ defined as
\begin{align}\label{eq:min_norm}
  \vec{\hat{z}} :=& (\bPhi(s)-\delta \S)^\# \delta \S \vec{\hat{w}}_u(s) \nonumber\\
  =& \begin{bmatrix}
    (\bPhi_{11}(s)-\delta \S_{11})^{-1} \delta \S_{11} \!\!\! & (\bPhi_{11}(s)-\delta \S_{11})^{-1} \delta \S_{12} \\ 0 & 0
  \end{bmatrix}
  \begin{bmatrix} \vec{\hat{w}}_1 \\ \hat{c} \end{bmatrix} \nonumber\\
  =& \begin{bmatrix}
    (\bPhi_{11}(s)-\delta \S_{11})^{-1} \delta (\S_{11} \vec{\hat{w}}_1 +  \S_{12}\hat{c}) \\ 0
  \end{bmatrix}
\end{align}
is of the correct form $[\vec{\hat{z}}_1^T,0]^T$ and $\vec{\hat{z}}_1$ satisfies Eq.~\eqref{eq:z1}. With this concept, in lieu of Eq.~\eqref{e:scaling1}, we \emph{define}
\begin{equation}\label{eq:poundversion}
\T_{\vec{z},\vec{w}_u}^{u,\#}(s,\delta)
 =(\bPhi(s)-\delta \S)^\#\delta \S.
\end{equation}
\begin{lemma} \label{l:ma}
We have, in lieu of Eq.~\eqref{eq:MIL},
\begin{equation*}
  (\bPhi(s)-\delta \S)^\#\delta \S
  = (I-\bPhi(s)^\#\delta \S)^{\#}\bPhi(s)^\#\delta \S,
\end{equation*}
if $\bPhi_{11}(s)$ is invertible.
\end{lemma}

\begin{proof} This can be verified by direct computation. The LHS of the equation gives
\begin{multline*}
(\bPhi-\delta \S)^\# \delta \S
= \begin{bmatrix}
  (\bPhi_{11}-\delta \S_{11})^{-1} & 0 \\ 0 & 0
  \end{bmatrix}
  \begin{bmatrix}
  \delta \S_{11} & \delta \S_{12} \\ 0 & 0
  \end{bmatrix} \\
= \begin{bmatrix}
  (\bPhi_{11} -\delta \S_{11})^{-1} \delta \S_{11} & (\bPhi_{11}-\delta \S_{11})^{-1}\delta \S_{12} \\ 0 & 0
  \end{bmatrix}.
\end{multline*}
For the RHS we note that
\begin{equation*}
\bPhi^\# \delta \S
= \begin{bmatrix}
  \bPhi_{11}^{-1} & 0 \\ 0 & 0
  \end{bmatrix}
  \begin{bmatrix}
  \delta \S_{11} \!\!& \delta \S_{12} \\ 0 & 0
  \end{bmatrix}
= \begin{bmatrix}
  \bPhi_{11}^{-1} \delta \S_{11} \!\! & \bPhi_{11}^{-1} \delta \S_{12} \\ 0 & 0
  \end{bmatrix}.
\end{equation*}
Thus, we have
\begin{equation*}
\I - \bPhi^\# \delta \S
= \begin{bmatrix}
  \I'-\bPhi_{11}^{-1} \delta \S_{11} & -\bPhi_{11}^{-1} \delta \S_{12} \\ 0 & 1
  \end{bmatrix}.
\end{equation*}
Furthermore, $(\I'-\bPhi_{11}^{-1} \delta \S_{11})^{-1} \bPhi_{11}^{-1} = (\bPhi_{11}-\delta\S_{11})^{-1}$ as $B^{-1} A^{-1}=(AB)^{-1}$, and thus
\begin{align*}
  &(\I-\bPhi^\# \delta \S)^{\#} \bPhi^{\#} \delta \S\\
  &=\begin{bmatrix}
      (\I'-\bPhi_{11}^{-1} \delta \S_{11})^{-1} & 0 \\ 0 & 0
    \end{bmatrix}
    \begin{bmatrix}
     \bPhi_{11}^{-1} \delta \S_{11} & \bPhi_{11}^{-1} \delta \S_{12} \\
     0 & 0
    \end{bmatrix} \\
&= \begin{bmatrix}
     (\bPhi_{11}-\delta \S_{11})^{-1} \delta \S_{11} &
     (\bPhi_{11}-\delta \S_{11})^{-1} \delta \S_{12} \\
     0 & 0
   \end{bmatrix}
\end{align*}
which agrees with the LHS.
\end{proof}

With the \#-inverse concept, we can now rewrite Eq.~\eqref{e:Gzwu} as
\begin{equation}
  G_{\vec{z},\vec{w}_u}(s) = \begin{pmatrix}\bPhi(s)^{\#}\S & \bPhi(s)^{\#}\S\\ \I & 0 \end{pmatrix}.
\end{equation}
Along with the feedback $\vec{\hat{\eta}}=(\delta I) \vec{\hat{v}}$ this reproduces the \#-inverse version~\eqref{eq:poundversion} of $\T_{\vec{z},\vec{w}_u}^u(s,\delta)$, which coincides with its classical inverse version of Eqs.~\eqref{e:scaling1}-\eqref{eq:MIL}.

\begin{theorem}\label{t:continuity}
If $\bPhi_{11}^{-1}(s)$ exists at $s=0$, $G_{\vec{z},\vec{w}_u}(s)$ is continuous, including at $s=0$. Moreover,
$\mu_\mathcal{D}(G_{\vec{z},\vec{w}_u}(s))$ is continuous, including at $s=0$, provided $0$ is not a critical value of the $s$-parameterized mapping $f_s : \bm\Delta \mapsto \det(I+G_{\vec{z},\vec{w}_u}(s)\bm\Delta)$.
\end{theorem}

\begin{proof}
The first statement is trivial. The proof of the second statement follows the same lines as~\cite[Th. 23.25]{Jonckheere1997}. From the implicit function theorem, it follows that $f_\omega^{-1}(0)$ is a differentiable manifold; call it $X_\omega$. $1/\mu_\mathcal{D}(\imath \omega)$ is the minimum $\|\bm\Delta\|$ such that $\bm\Delta \cap X_\omega \ne \emptyset$. Moreover, under variation of $\omega $, $X_\omega$ is deformed by an isotopy. Hence, $1/\mu_\mathcal{D}(\imath\omega)$ and therefore $\mu_\mathcal{D}(\imath\omega)$ is continuous relative to $\omega$.
\end{proof}
Continuity of $G_{\vec{z},\vec{w}_u}(s)$ and $\T^u_{\vec{z}, \vec{w}_u}(s,\delta)$ is salvaged by defining the \#-inverse to generate Eq.~\eqref{eq:min_norm} as the solution of Eq.~\eqref{eq:solution}.

\subsection{Initial State Preparation Error Response} \label{sec:sing_prep_error}

As before, the last ($N^2$-th) row of $\A$ and $\S$ vanishes and $z_{N^2}(0)=0$
as any prepared state
must be represented by a density of trace $1$. Therefore, Eq.~\eqref{eq:x0m1} gives $\vec{\hat{z}}(s) =
\T_{\vec{z},\vec{w}_u}^{u,\#}(s) \vec{\hat{w}}_u+ \T_{\vec{z},\vec{z}_0}^{u,\#} \vec{z}(0)$ and ${T}_{\vec{z},\vec{z}_0}^{u,{\#}}(s) :=(\bPhi-\delta \S)^\#$.

\begin{lemma}
Under the same assumptions as in Lemma~\ref{l:ma} the following matrix \#-inversion lemma holds
\begin{equation}
  (\bPhi-\delta \S)^\# = \bPhi^\#+\bPhi^\#\delta \left(\I-\S \bPhi^\# \delta \right)^{\#} \S \bPhi^\#.
\end{equation}
\end{lemma}

\begin{proof}
We need to show that the matrix inversion lemma holds when we replace the regular inverse by the \#-inverse. $\A^\# \bm B^\# = (\bm B\A)^\# = \diag((\bm B_{11}\A_{11})^{-1},0)$ still holds and
\begin{align*}
 \text{LHS: } &(\bPhi-\delta \S)(\bPhi-\delta \S)^{\#} = \I' \\
 \text{RHS: } &(\bPhi-\delta \S) \bPhi^{\#} +(\bPhi-\delta \S)(\bPhi-\delta \S)^{\#} \delta \S \bPhi^{\#} \\
              &= \I' - \delta \S \bPhi^{\#} + \I'\delta \S \bPhi^{\#} = \I'.
\end{align*}
\end{proof}
It follows from the lemma that ${T}_{\vec{z},\vec{z}_0}^{u,\#}(s)$ can be represented as
\begin{equation}\label{eq:newG}
  \begin{pmatrix} \vec{\hat{v}}\\\vec{\hat{z}_{\vec{z}_0}}(s)\end{pmatrix} =
  \underbrace{\begin{pmatrix}\S\bPhi(s)^{\#} & \S\bPhi(s)^{\#}\\
  \bPhi(s)^{\#} & \bPhi(s)^{\#} \end{pmatrix}}_{=:G_{\vec{z},\vec{z}_0}^{(\#)}(s)}
  \begin{pmatrix} \vec{\hat{\eta}} \\ \vec{z}_0\end{pmatrix}
\end{equation}
with the feedback $\vec{\hat{\eta}} =(\delta \I) \vec{\hat{v}}$ wrapped around it. Eq.~\eqref{eq:newG} provides the substitute for Eq.~\eqref{e:Gzz0} for all $s$ and leads to ${\mu}_\mathcal{D}^{(\#)}$, as shown in Fig.~\ref{fig:S1357x0}.

\begin{remark}
Note that $(\cdot)^\#$ is not the Moore-Penrose pseudo-inverse since $(\bPhi-\delta \S)^\#(\bPhi-\delta \S)$ is not Hermitian.
\end{remark}
\begin{remark}
The \emph{matrix pseudo-inversion lemma}~\cite{matrix_pseudo_inversion_lemma, operator_pseudo_inversion_lemma}, $(\bPhi-\delta \S)^+ = \bPhi^++\bPhi^+\delta \left(\I-\S \bPhi^+ \delta \right)^+ \S \bPhi^+$, where $(\cdot)^+$ is the Moore-Penrose pseudo-inverse, is not applicable here as it requires $\A$ and $\S$ to be Hermitian, non-negative and satisfy restrictive range conditions.
\end{remark}

\begin{table}
\caption{Real, finite generalized eigenvalues of the pair $(\A,-\S_k)$. For our nominal system parameters any value of $\delta$ is admissible for $k=1,\ldots,4$. To avoid negative decoherence rates we must have $\delta\ge -1$ for $k=5$ and $\delta\ge 0$ for $k=6,7$.}\label{t:gen_eig}
\centering
\begin{tabular}{|r|l|}\hline
coefficient matrices & real, finite generalized eigenvalues
\\\hline
$\A+\delta \S_1$, $\A+\delta \S_2$ & none \\
$\A+\delta \S_3$, $\A+\delta \S_4$ & $\delta=-0.2$ (double)\\
$\A+\delta \S_5$                   & $\delta=-1$ (double)\\
$\A+\delta \S_6$, $\A+\delta \S_7$ & $\delta=-0.0057$, $-0.6346$, $-1.0462$, $-2.6465$
\\\hline
\end{tabular}
\end{table}

\section{Application: Two Qubits in a Cavity}\label{sec:2qubits}

We apply the method of perturbation-to-error transmission to two two-level atoms in a lossy cavity designed to maximize entanglement generation between the atoms~\cite{motzoi}, or more broadly between two quantum devices in the quantum Internet~\cite{quantum_internet_van_meter}. After adiabatic elimination of the cavity via a unitary transformation~\cite{motzoi}, the dynamics can be described by
\begin{equation}\label{eq:normalizedV}
  \tfrac{d}{dt}{\rho}(t) = -\imath [H_{\alpha,\Delta},\rho(t)]
  + \sum_k \gamma_k^2 \mathfrak{L}\left(\sigma_{-}^{(k)}\right) \rho(t).
\end{equation}
Denoting the raising operator by $\sigma_+=\begin{pmatrix}0 & 0\\1 & 0\end{pmatrix}$, the lowering operator by $\sigma_-:=\sigma_+^\dagger$, and defining the operators $\sigma_{\pm}^{(1)}=\sigma_{\pm}\otimes I_{2 \times 2}$ and $\sigma_{\pm}^{(2)}=I_{2 \times 2}\otimes \sigma_{\pm}$ as usual, the Hamiltonian is
\begin{align}
  H_{\alpha,\Delta}
  &=\sum_{n=1}^2\left( \alpha_n^*\sigma_+^{(n)} + \alpha_n \sigma_-^{(n)}
                       +\Delta_n \sigma_+^{(n)}\sigma_-^{(n)} \right) \nonumber \\
  &= \begin{pmatrix}
           0          & \alpha_2   & \alpha_1   & 0 \\
           \alpha_2^* & \Delta_2   & 0          & \alpha_1 \\
           \alpha_1^* & 0          & \Delta_1   & \alpha_2\\
           0          & \alpha_1^* & \alpha_2^* & \Delta_1 + \Delta_2
     \end{pmatrix},\label{eq:HV}
\end{align}
where $\alpha_1$, $\alpha_2$ are the driving fields and $\Delta_1$, $\Delta_2$ are the detuning parameters. The decoherence in Eq.~\eqref{eq:normalizedV} will be generalized by replacing it by the super-operator $\mathfrak{L}(V_\gamma)$ where
\begin{equation}
  V_\gamma = \sum_{n=1}^2 \gamma_n \sigma_-^{(n)}
  = \begin{pmatrix}
    0 & \gamma_2 & \gamma_1 & 0\\
    0 &  0 & 0 & \gamma_1 \\
    0 &  0 & 0 & \gamma_2 \\
    0 &  0 & 0 & 0
    \end{pmatrix}.
\end{equation}
Note that $[H_{\alpha,\Delta},V_\gamma]\ne 0$. To examine the system's robustness, we consider the system dynamics in the Bloch formulation,
\begin{equation}\label{e:Bloch_cavity}
  \tfrac{d}{dt} \vec{r}_p(t) = (\A_{\alpha,\Delta,\gamma}+\delta \S(\alpha_1, \alpha_2;\Delta_1,\Delta_2; \gamma_1,\gamma_2)) \vec{r}_p(t).
\end{equation}
This is relative to the Pauli basis $\{e_k \otimes e_\ell: k,\ell=1,\dotsc,4\}$, where $(e_1,e_2,e_3,e_4)=\frac{1}{\sqrt{2}}(I_{2 \times 2}, \sigma_x,\sigma_y,\sigma_z)$, with
\begin{equation}\label{e:Pauli}
  \sigma_{x}=\begin{pmatrix}0 & 1\\1& 0\end{pmatrix},
  \sigma_{y}=\begin{pmatrix}0 & \imath\\-\imath& 0\end{pmatrix},
  \sigma_{z}=\begin{pmatrix}-1 & 0\\0& 1\end{pmatrix}.
\end{equation}
The Pauli operators and components are ordered so that $\vec{z}_{16}$ is the error on $\Tr(\rho)$ and, hence, vanishes. $\A_{\alpha,\Delta,\gamma}$ is a real $16 \times 16$ matrix whose last row vanishes and whose last column depends exclusively on $\gamma$ but does not vanish for $\gamma>0$. It can further be verified that for $\alpha\neq 0$, $\gamma\neq0$ and $\Delta\neq 0$, the rank of $\A_{\alpha,\Delta,\gamma}$ is $15$, and the eigenvalues of $\A_{\alpha,\Delta,\gamma}$ have negative real parts except for one $0$ eigenvalue due to the trace constraint for $\rho$. Generally, we have stability for non-zero detuning but for $\Delta=0$ the rank drops to $14$, implying the existence of a one-dimensional subspace of steady-states.

$\A_{\alpha,\Delta,0}$ corresponds to unitary evolution and hence its eigenvalues are purely imaginary. Generically, its rank is $12$, except for special cases such as $\alpha_1=\alpha_2$ and $\Delta_1=-\Delta_2$ when $\rank(\A_{\alpha,\Delta,0})=10$. Thus the decoherence $\A_{0,0,\gamma}$ acts as a stabilizing controller for the plant $\A_{\alpha,\Delta,0}$ with state feedback control $\vec{u}(r(t)) = \A_{0,0,\gamma} \vec{r}(t)$, in accordance with Eq.~\eqref{eq:fb}.

\begin{figure} \centering
  \includegraphics[width=.9\columnwidth]{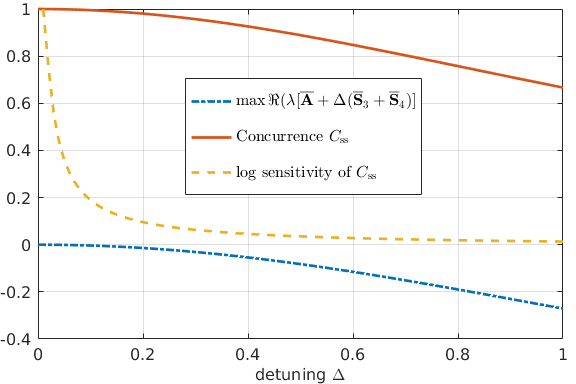}
  \caption{Maximum of the real part of the eigenvalues $\lambda$ of $\overline{\A}+\Delta (\overline{\S}_3+\overline{\S}_4)$, concurrence of steady-state and log-sensitivity of steady-state concurrence as function of detuning $\Delta$ for $\alpha=\gamma=1$. All three figures of merit are concordant, i.e. they decrease with increasing detuning.} \label{fig:detuning}
\end{figure}

\subsection{Structured Uncertainties and Frequency Response}\label{sec:robustness}

The \emph{structure} of the perturbation, $\S(\alpha_1, \alpha_2;\Delta_1,\Delta_2; \gamma_1,\gamma_2)$, of Eq.~\eqref{e:Bloch_cavity} is normalized as $\alpha_k=0,1$, $\Delta_k=0,\pm 1$, $\gamma_k=0,1$. Specifically, we distinguish the following cases:
\begin{subequations}\label{eq:S1234567}
\begin{align}
  \S_1 & = \S(1,0;0,0;0,0), &\S_2 & = \S(0,1;0,0;0,0),\\
  \S_3 & = \S(0,0;1,0;0,0), &\S_4 & = \S(0,0;0,-1;0,0),\\
  \S_5 & = \S(0,0;0,0;1,1), \\
  \S_6 & = \S(0,0;0,0;1,0), & \S_7 & = \S(0,0;0,0;0,1).
\end{align}
\end{subequations}
$\S_5$ corresponds to \emph{collective} dissipation, while $\S_6$, $\S_7$ are structured perturbations corresponding to \emph{single} qubit spontaneous emission. For structures $\S_6,\S_7$, the sizes are $\delta=\gamma_1^2,\gamma^2_2$, resp.
Due to the structure of $\S_5$ this does not reduce to the case of a single, lossy qubit, but rather a special case of a two-qubit system with coupling mediated by a cavity.

The $\|\bPhi^\#(s)\|$ and the error gain $\norm{\T_{\vec{z},\vec{w}_u}^{u}(s,\delta\S_k)}$ for the selected perturbations as functions of frequency $s=\imath\omega$ are shown in Fig.~\ref{fig:2qubit_example_gain} for nominal plant and controller parameters $\alpha_1=\alpha_2=1$, $\gamma_1=\gamma_2=1$ and $\Delta_1 = -\Delta_2 = \Delta = 1$ and different values of $\delta$. Due to symmetry, the effects of $\S_1$ and $\S_2$ are the same, and similarly for $\S_3$ and $\S_4$, and $\S_6$ and $\S_7$, respectively. Hence, it suffices to consider four perturbations. Except for $\S_5$, the bound of Eq.~\eqref{eq:bound} for dephasing in the Hamiltonian basis is violated here, as expected, as this system is not a system with dephasing in the Hamiltonian basis. We focus on sensitivity to low-frequencies as we expect low-frequency drift in the detuning and coupling parameters and $1/f$ noise, e.g., due to laser flicker noise in atomic clock systems~\cite{nature_photonics_clock} or magnetic flux noise in superconducting qubits~\cite{NIST_1_over_f}.

Plotting $\norm{\bPhi^{\#}(\imath\omega)}$ versus frequency in Fig.~\ref{fig:2qubit_example_gain}(a) shows it is maximal for $\omega=0$. The figure also shows that the \#-inverse gives different values from the Moore-Penrose inverse at $s=0$. For non-zero $\delta$, the gain $\norm{\T_{\vec{z},\vec{w}_u}^u( \imath\omega,\delta\S_k)}$ need not be maximal at $\omega=0$ as illustrated in Fig.~\ref{fig:2qubit_example_gain}(b,c). Although, Fig.~\ref{fig:2qubit_example_gain2}(b) suggests that for most perturbations the maximum gain is still achieved for $\omega=0$. The maximum of $\norm{\T_{\vec{z}, \vec{w}_u}^{u} (\imath\omega,\delta\S_k)}$ depends on the perturbation $\S_k$ and strength $\delta$. Fig.~\ref{fig:2qubit_example_gain2}(a) suggests that the system is most sensitive to perturbations $\S_3$ in the detuning. Other work~\cite{motzoi} suggests that robust solutions can be found outside of the regime where $\Delta_1 =-\Delta_2$, but such investigations are outside of the scope of this work.

\subsection{Bounding Dynamic Perturbation Transmission}\label{sec:mu}

Another way to assess robustness against parameter variation is to examine the structured singular value $\mu_{\mathcal{D}}$ bounding the error transmission as $\|\T_{\vec{z} \vec{w}}^{u}(s,\delta\S)\| \leq \mu_{\mathcal{D}}(s)$,
for $\delta < 1/\mu_{\mathcal{D}}(s)$ as made precise by Th.~\ref{th:robust_performance}. The results of Sec.~\ref{sec:pseudo_inversion} apply if the submatrices $\bPhi_{11}$ and $\S_{11}$ of $\bPhi(0)$ and $\S$ are invertible. The $16 \times 16$ matrices $\A_{\alpha,\Delta,\gamma}$ and $\A_{\alpha, \Delta, \gamma}+\delta \S_k$ of Eq.~\eqref{e:Bloch_cavity}, evaluated at $\alpha_1=\alpha_1=1$, $\Delta_1 =-\Delta_2=0.1$ and $\gamma_1=\gamma_1=1$, are singular with rank $15$. $\delta$-generically $\S_k$ has rank $15$ for $k \neq 5$, and rank $14$ for $k=5$. The nongeneric $\delta$-values are computed as generalized eigenvalues of the pair $(\A,-\S_k)$ and are displayed in Table~\ref{t:gen_eig}.

The minimum structured ``destabilizing'' perturbation $\bm \Delta$ need not be unique as it is easily seen from Eq.~\eqref{eq:AS} that $G_{\vec{z},\vec{w}_u}(0)$ has one vanishing row and one vanishing column, which causes the solution to $\det(I+G_{\vec{z},\vec{w}_u}(0)\bm \Delta)=0$ to have a completely arbitrary row and column but the size $\|\bm \Delta\|$ is uniquely defined.

Fig.~\ref{fig:S1357_all}(a) shows simulation results for the structured uncertainties $\S_1, \S_3, \S_5, \S_7$ as a function of frequency on a frequency scale comparable with that of Fig.~\ref{fig:2qubit_example_gain}. Simulation results for the structured singular value $\mu$ as $s$ decreases to $0$ along the real axis are shown in Fig.~\ref{fig:S1357_all}(b). Except for $\S_1$ they show continuity of $\mu_D$ and the discrepancy between the upper and lower bounds is very mild (not visible on a log-scale). The sensitivity for $\S_1$ indicates that asymmetric perturbation of the driving fields is detrimental to entanglement generation. This argument is strengthened as the behavior disappears if $\S_1$ is replaced by a symmetric perturbation of the driving fields,  $\S(1,1;0,0;0,0)$ (not shown). Fig.~\ref{fig:S1357_all}(c) shows the upper bounds on $\mu_\mathcal{D}^{(\#)}$ for initial state preparation error transmission (Sec.~\ref{sec:sing_prep_error}) for $s\to 0$ along the real axis.


\subsection{Concurrence and its Log-Sensitivity}\label{sec:concurrence}

Two-qubit entanglement can be measured by the \emph{concurrence} of the two-qubit density operator $\rho$~\cite{EntanglementConcurrence},
\begin{equation}
  C(\rho)=\max\{0,\lambda_1-\lambda_2-\lambda_3-\lambda_4\},
\end{equation}
where $\lambda_k$ are the eigenvalues, in decreasing order, of $\sqrt{\sqrt{\rho}\tilde{\rho}\sqrt{\rho}}$ with $\tilde{\rho}=(\sigma_y \otimes \sigma_y)\mathrm{conj}(\rho)(\sigma_y \otimes \sigma_y)$, 
and $\sigma_x$, $\sigma_y$, and $\sigma_z$ are the Pauli operators.

For anti-symmetric detuning, $\Delta_1=-\Delta_2=\Delta$, symmetric driving, $\alpha_1=\alpha_2=\alpha$, symmetric dissipation $\gamma_1=\gamma_2=\gamma$, and $\Delta$, $\alpha$ and $\gamma$ real, i.e., in the case considered here,
\begin{equation*}
  \ket{\Psi_{\ss}} = \left(1/\sqrt{\Delta^2 + 2\alpha^2}\right)
  [\Delta,\alpha ,-\alpha, 0]^T
\end{equation*}
is a steady-state of the system: $H_{\alpha,\Delta}\ket{\Psi_{\ss}} = \vec{0}$ and $V_\gamma\ket{\Psi_{\ss}}=\vec{0}$. Thus, $\rho_{\ss} =\ket{\Psi_{\ss}} \bra{\Psi_{\ss}}$ satisfies $\tfrac{d}{dt}{\rho}_{\ss}=0$. The concurrence of this steady-state is~\cite{Wang2010}
\begin{equation*}
  C_{\ss}:=C(\rho_{\ss}) = \left[\tfrac{1}{2}(\Delta/\alpha)^2+1 \right]^{-1}.
\end{equation*}
Since this steady-state is generically globally attractive, any initial state converges to it, and its concurrence determines the performance. To maximize the concurrence, we want $\Delta/\alpha$ as small as possible, but in the limit of no detuning, $\Delta\to0$, the attractivity of the steady-state is lost. So there are trade-offs in the speed of convergence and robustness.

If we measure the performance of the control scheme by the concurrence $C_{\ss}$ of the steady-state and compare it with the other measures shown in Fig.~\ref{fig:detuning}, there is concordance~\cite{statistical_control} between $C_{\ss}$ and its log-sensitivity, i.e. they both decrease with increasing detuning. The concurrence error $1-C_{\mathrm{ss}}$ increases while its log-sensitivity decreases. This is in agreement with the classical conflict between the sensitivity function $S$ and its log-sensitivity, the complementary sensitivity $T$, for which $S+T=1$ holds. The concordance between $C_{\ss}$ and the stability margin measured as
$\left|\max_n \{\Re \lambda_n(\overline{\A}+\Delta(\overline{\S}_3+\overline{\S}_4))\}\right|$ is also classical as the lower the concurrence performance, the higher the stability margin.

\section{Conclusion}

We have developed a robust performance formalism for controlled open quantum systems subject to a variety of structured uncertainties, ranging from uncertain parameters in the Hamiltonian to uncertainties caused by initial state preparation errors and decoherence. The existence of closed-loop poles at $0$ in the Bloch equation violates the traditional closed-loop stability requirement. The formalism addresses these issues and allows quantification of the transmission of the dynamic disturbance to quantum state error subject to structured uncertainties while addressing the continuity of $\mu_{\mathcal{D}}$. Proceeding from the general Lindblad equation gives the formalism wide physical applicability. The cavity case-study reveals that quantum control requires enlarging the concept of a ``performance measure'' to include novel measures such as entanglement concurrence, which may have properties unconventional in control such as being nonlinear in the state. While robustness of a nonlinear performance remains to be developed, the cavity example suggests a trade-off between the concurrence and its log-sensitivity due to decoherence, as expected in conventional robust control. Earlier results already show that coherent control escapes some of the classical limitations~\cite{statistical_control}, but under decoherence, classicality re-emerges~\cite{CDC_decoherence}. It is an open question in quantum control whether coherent control allows objectives that are traditionally conflicting to co-exist in general~\cite{statistical_control, CDC_phase, Edmond_IEEE_AC, soneil_mu}. Additionally, the development of more general tools and methods for quantum control that are not derived from existing classical methods requires further work.

\appendix\label{a:proof}
\setcounter{definition}{0}
\numberwithin{definition}{subsection}
\setcounter{theorem}{0}
\numberwithin{theorem}{subsection}
\setcounter{corollary}{0}
\numberwithin{corollary}{subsection}
\setcounter{lemma}{0}
\numberwithin{lemma}{subsection}

\subsection{Multiple Eigenvalues}\label{a:multiple}

Let $\mathcal{H}(N)$ be the set of $N \times N$ Hermitian matrices.

\begin{definition}
A property $\mathfrak{P}$ of a set of matrices $\mathcal{M}$ is \emph{generic} if the subset of matrices where it holds is open and dense in $\mathcal{M}$ for a relevant topology on $\mathcal{M}$. A submanifold $\mathfrak{M}$ is $\mathbb{R}^*$-homogeneous if $\mathbb{R}^*\mathfrak{M}=\mathfrak{M}$.
\end{definition}

\begin{theorem}~\cite[Corollary 4.12]{GutkinJonckheereKarow}.
For $N \geq 2$, the subset $\mathfrak{M}_{n_1,n_2,\dotsc,n_{\bar{N}}}$ of $\mathcal{H}(N)$ with eigenvalue multiplicities $n_1,n_2,\dotsc,n_{\bar{N}}$ is a $\mathbb{R}^*$-homogeneous sub-manifold of codimension $\left(\sum_{k=1}^{\bar{N}} n_k^2\right)-\bar{N}$ in $\mathcal{H}(N)$. The subset $\mathfrak{V}$ of $\mathcal{H}(N)$ with multiple eigenvalues is a real algebraic variety of codimension $3$ in $\mathcal{H}(N)$.
\end{theorem}

\begin{corollary}
The property ``no multiple eigenvalues'' is generic in $\mathcal{H}(N)$.
\end{corollary}

\begin{corollary}\label{c:3_parameters}
Let $H_\theta$ be a family of Hermitian matrices in $\mathcal{H}(N)$ that depends continuously on the real parameter $\theta$. Let $H_0 \in \mathcal{H}(N) \setminus \mathfrak{V}$ be the subset of matrices that has no multiple eigenvalues. Then, generically, a 3D (real) perturbation $\theta=(\theta_1, \theta_2, \theta_3)$ is necessary to reach multiple eigenvalues. Under nongeneric conditions, more parameters are needed, unless there exists a unique $k^*$ such that $n_{k^*}=2$ and $n_{k \ne k^*}=1$, in which case three parameters still suffice.
\end{corollary}

\begin{proof}
Consider $H_{\theta^*}=\arg \min_{H_\theta\in \mathfrak{V}}d(H_\theta,H_0)$, defining a projection $\pi: \mathcal{H}(N) \to \mathfrak{V}$ orthogonal to the stratum of $\mathfrak{V}$ that contains $H_0$. Assume $H_{\theta^*}$ is a differentiable point of $\mathfrak{V}$. Let $\vec{\sigma}_{\mathfrak{V}}$ be an orthonormal basis of the tangent space of $\mathfrak{V}$ at $H_{\theta^*}$ and complete it to an orthonormal basis $\{\vec{e}_{\mathfrak{V}},\vec{\sigma}_1,\vec{\sigma}_2,\vec{\sigma}_3\}$ of $\mathcal{H}(N)$. Then the coordinates of $(H_{\theta^*}-H_0)$ relative to $\{\vec{\sigma}_1,\vec{\sigma}_2,\vec{\sigma}_3\}$ are the three parameters necessary to reach multiple eigenvalues. This situation is generic in $\mathfrak{V}$. Since $H_{\theta^*}$ is differentiable in $\mathfrak{V}$, there exists a neighborhood $\mathcal{N}_{H_{\theta^*}}$ where $\mathfrak{V}$-genericity remains valid. Since $\pi$ is continuous, $\pi^{-1}(\mathcal{N}_{H_{\theta^*}})$ is a neighborhood of $H_0$ where the projection is differentiable; hence $\mathfrak{V}$-genericity. If $H_{\theta^*}$ is a singular point, it belongs to a manifold $\mathfrak{M}_{n_1,n_2,\dotsc,n_{\bar{N}}}$ of the algebraic variety $\mathfrak{V}$. Construct a basis $\vec{\sigma}_{\mathfrak{M}}$ of the tangent space, and complete it to a basis of $\mathcal{H}(N)$, viz., $\vec{\sigma}_{\mathfrak{M}},\vec{\sigma}_1,\dotsc,\vec{\sigma}_{\left(\sum_{k=1}^{\bar{N}} n_k^2\right)-\bar{N}}$. Clearly, $(\sum_{k=1}^{\bar{N}} n_k^2)-\bar{N}\geq 3$ parameters are needed, with equality only if a unique $k_* n_{k^*}=1$; for all other $n_k=1$ (see~\cite[p.\ 162]{GutkinJonckheereKarow} for details regarding that last inequality).
\end{proof}

This result can be traced back to~\cite{von_neumann_wigner} but we have clarified the ``in general'' on ~\cite[p. 553]{von_neumann_wigner}. The most general singularities in $\mathfrak{V}$ are those of the smallest codimension, i.e., $3$. Such an eigenstructure is \emph{in general} unstable under perturbation as eigenvalues with multiplicity $>1$ split into lower multiplicity eigenvalues under universal unfolding~\cite{CastrigianoHayes1993}. Securing stability of the eigenstructure of $H$ (and $V$) requires $\bar{N}=N$, which can be justified invoking \emph{genericity}. See also~\cite{GutkinJonckheereKarow,adiabatic}.

The argument can be reversed to split a multiple eigenvalue into simple eigenvalues under the $\mathfrak{V}$-generic condition if we have three uncertain parameters. The coupled four-qubit system has $16$ parameters in its $4 \times 4$ Hamiltonian, all of which are uncertain to some degree. If we consider the simplest case of exactly one double eigenvalue, then $\mbox{codim}(\mathfrak{M}_{n_1=2,n_2=n_3=1})=3$, i.e., the double eigenvalue can be split with three parameters. Any higher multiplicity structure of the eigenvalues (precisely, $\bar{N}<3$) would create $\mbox{codim}(\mathfrak{M}_{n_1, n_2, \dotsc, n_{\bar{N}}})>3$ and more than three parameters would be needed to achieve an arbitrary splitting. Simply put, given the high number of uncertain parameters in quantum systems, the ``no multiple eigenvalues'' assumption is reasonable.

\subsection{Dole\v{z}al's Theorem and its Extension}\label{a:Dolezal}

The $16 \times 16$ Bloch state transition matrix $\A_\theta$ in the cavity example in Sec.~\ref{sec:2qubits} has generic rank $15$,
dropping for some $\theta$s, raising the question of whether the eigendecomposition can be made continuous in $\theta=[\alpha_1,\alpha_2,\Delta_,\Delta_2,\gamma_1,\gamma_2]$.

\begin{theorem}[Dole\v{z}al~\cite{Dolezal1}]\label{t:Dolezal}
If $\A_{\theta}$ is continuous in $\theta$ with constant rank $r$ for a subset $\Theta_0 \subset \Theta$, there exist $M_\theta$, $B_\theta$ continuous in $\theta$ with $M_\theta$ nonsingular such that $\A_{\theta}M_\theta=(B_\theta,0)$ where $B_\theta$ has $r$ columns.
\end{theorem}

\begin{corollary}\label{c:orthonormal}
Under the same conditions as in Th.~\ref{t:Dolezal}, the basis of the null space can be continuous and orthonormal.
\end{corollary}

\begin{proof}
Let $\{e_1,\dotsc,e_n\}$ be a basis depending continuously on a parameter $\theta$. Take the first vector $e_1$, $\|e_1\|=1$, of the orthonormalized basis and consider $P_{e_1^\perp}\{e_2,\dotsc,e_n\}$. Assume the latter $n-1$ vectors are not independent. Then for some $\alpha_2,\dotsc,\alpha_n$ not identically vanishing, we would have $P_{e_1^\perp}\left(\sum_{i=2}^n \alpha_i e_i\right)=0$. Define $\beta_i$ such that $\beta_i e_1=(I-P_{e_1^\perp})\alpha_i e_i$. Hence, $P_{e_1^\perp} \left(\sum_{i=2}^n \alpha_ie_i\right) +\left(\sum_{i=2}^n\beta_i \right)e_1=\sum_{i=2}^n\alpha_ie_i$, and finally, $-\left(\sum_{i=2}^n \beta_i\right)e_1+\sum_{i=2}^n\alpha_i e_i =0$, which contradicts the linear independence of $\{e_1,\dotsc,e_n\}$. Orthonormalization proceeds by induction on the dimension. Continuity in $\theta$ follows from continuity of $P_{e_1(\theta)^\perp}$ on $\theta$.
\end{proof}

When the rank of $A_\theta$ changes, we have the following:

\begin{theorem}[Silverman \& Bucy~\cite{Dolezal2}]\label{t:SilvermanBucy}
If $\A_{\theta}$ is analytic in $\Theta_\omega$ where $\mbox{rank}(A_\theta)\leq r$, there exist analytic $M_\theta$, $B_\theta$ with $M_\theta$ nonsingular such that $\A_{\theta}M_\theta=(B_\theta,0)$ and $B_\theta$ has $r$ columns.
\end{theorem}

This result does not hold in general without analyticity. Dole\v{z}al's theorem was originally proved over a real interval but can be extended to multiple parameters by a partition of unity argument~\cite[p. 72]{primer_analytic}. Partitions of unity do not exist in the analytic category~\cite[pp. 174-175]{primer_analytic} so analytic extension is restricted to a neighborhood of $\theta_*$ where a rank changes. Such results can be extended in the continuous case to a basis of the invariant subspace of any eigenvalue by replacing $\A_\theta$ by $\A_\theta-\lambda_\theta I$, invoking continuity of the eigenvalues. In the analytic case, Weierstrass' preparation theorem~\cite[Sec. 6.1]{primer_analytic},\cite[Sec. D.1]{Jonckheere1997} shows $\lambda_\theta$ is analytic along the analytic branches through the multiple eigenvalue with guaranteed analytic basis of the eigenspace.

\subsection{Dephasing in the Hamiltonian Eigenbasis}\label{a:DephasingHamiltonianEigenbasis}

\begin{lemma}\label{l:PQQP}
Let $P$, $Q$ be Hermitian operators in $\H$ that commute. If the (orthonormal) bases of the eigenspaces of $P$ or $Q$ associated with the multiple eigenvalues are freely adjustable then $P$ and $Q$ are simultaneously diagonalizable by a unitary transformation. If they are constrained then $P$ and $Q$ are only simultaneously \emph{block} diagonalizable via a unitary. The results remain valid if $P$ and $Q$ depend continuously on a parameter $\theta$ subject to eigenvalues of constant multiplicities.
\end{lemma}

\begin{proof}
Let $P\V_{\lambda_i}=\lambda_i\V_{\lambda_i}$ and $Q\W_{\mu_j}=\mu_j\W_{\mu_j}$ be the eigenvalues $\lambda_i, \mu_j$ and orthonormalized eigenbases $\V_{\lambda_i}, \W_{\mu_j}$ of $P$ and $Q$. $PQ=QP$ yields
\begin{equation}\label{e:difficulty}
  Q\V_{\lambda_i}\subseteq \{\V_{\lambda_i}\}, \quad P\W_{\mu_j}\subseteq \{\W_{\mu_j}\},
\end{equation}
where $\{\V_{\lambda_i}\}$, $\{\W_{\mu_j}\}$ denote the \emph{subspaces} spanned by the (orthonormal) columns of $\V_{\lambda_i}$ and $\W_{\mu_j}$, respectively. It follows from Eq.~\eqref{e:difficulty} that $\W_{\mu_j}$ is an invariant subspace of $P$ and therefore must consist of eigensubspaces of $P$. To express $\W_{\mu_j}$ in such eigensubspaces of $P$, choose a set $I(j)$ of $i$-indices such that $\{\W_{\mu_j}\} \subseteq \oplus_{i \in I(j)} \{ \V_{i} \}$. In each $\{ \V_{\lambda_{i}} \}$, choose a (possibly reduced) basis $\bar{\V}_{\lambda_{i}}$ such that
\begin{equation}\label{e:oWofV}
  \{\W_{\mu_j}\}=\oplus_{i \in I(j)} \{\bar{\V}_{\lambda_{i}}\}.
\end{equation}
$\oplus_j \{ \W_{\mu_j}\}=\mathbf{H}$ implies $\oplus_j\oplus_{i \in I(j)} \{\bar{\V}_{\lambda_{i}}\}=\mathbf{H}$. To transform the \emph{subspace} equality to a \emph{basis} equality, let $\boxplus_i \V_{\lambda_i}=(\V_{\lambda_1},\V_{\lambda_2}, \dotsc )$ denote \emph{basis} rather than \emph{subspace} addition.

\emph{Case 1.} Assume both bases $\boxplus_j \W_{\mu_j} $ and $\boxplus_i\V_{\lambda_i}$ are given. To refine the subspace equality to an equality between bases, choose a rotation $R_{\mu_j}$ such that
\begin{equation}\label{e:WofV}
  \W_{\mu_j}=\left(\boxplus_{i \in I(j)} \bar{\V}_{\lambda_{i}}\right)R_{\mu_j}.
\end{equation}

Rewriting the eigenvector/eigenvalue equations for $P$ and $Q$ in terms of $\bar{\V}_{\lambda_i}$ using Eq.~\eqref{e:WofV} yields
\begin{align*}
  P \mathopen{}\mathclose{\left( \boxplus_j\boxplus_{i\in I(j)} \bar{\V}_{\lambda_i} \right)}
  &= \left(\boxplus_j \boxplus_{i\in I(j)} \bar{\V}_{\lambda_i} \right)\diag\{\Lambda_{I(j)}\},\\
  Q \mathopen{}\mathclose{\left(\boxplus_j\left( \boxplus_{i\in I(j)} \bar{\V}_{\lambda_i}R_{\mu_j}\right)\right)}
  &= \left(\boxplus_j \left( \boxplus_{i\in I(j)} \bar{\V}_{\lambda_i}R_{\mu_j}\right)\right)\diag\{\mu_j\}.
\end{align*}
Moving $R_{\mu_j}$ to the RHS of the second equation yields
\begin{align*}
  P\mathopen{}\mathclose{\left( \boxplus_j\boxplus_{i\in I(j)} \bar{\V}_{\lambda_i} \right)}
  &= \left(\boxplus_j \boxplus_{i\in I(j)} \bar{\V}_{\lambda_i} \right)\diag\{\Lambda_{I(j)}\},\\
  Q\mathopen{}\mathclose{\left(\boxplus_j\left( \boxplus_{i\in I(j)} \bar{\V}_{\lambda_i}\right)\right)}
  &= \left(\boxplus_j \left( \boxplus_{i\in I(j)} \bar{\V}_{\lambda_i}\right)\right)\diag\{\mu_jR^\dagger_{\mu_j}\}.
\end{align*}
$\boxplus_j\boxplus_{i\in I(j)} \bar{\V}_{\lambda_i}$ simultaneously \emph{block}-diagonalizes $P$ and $Q$.

\emph{Case 2.} By choosing the basis $\V_{\lambda_i}$, Eq.~\eqref{e:WofV} is simplified to $\W_{\mu_j}=\boxplus_{i \in I(j)} \bar{\V}_{\lambda_{i}}$. The rotations are no longer needed and $P$ and $Q$ are simultaneously diagonalizable.

Th.~\ref{t:Dolezal} guarantees that the results remain valid if $P$ and $Q$ depend continuously on a parameter $\theta$ subject to eigenvalues of constant multiplicities.
\end{proof}

\begin{corollary}\label{c:kernel}
Under the same conditions as Lemma~\ref{l:PQQP}, the kernel of one operator equals the direct sum of selected invariant subspaces of the other. If one such invariant subspace corresponds to an eigenvalue $\ne 0$, the kernels of $P$ and $Q$ are not coincidental.
\end{corollary}

\begin{proof}
Setting $\mu_j=0$ in Eq.~\eqref{e:oWofV}, $\{\W_{\mu_j=0}\}$ becomes the kernel of $Q$ and $\ker(Q)=\oplus_{i \in I(j)} \{\bar{\V}_{\lambda_{I(j)}}\}$. Therefore, the $\ker(Q)$ is made up of some invariant subspaces of $P$. If one such invariant subspace has corresponding eigenvalue $\ne 0$, the two kernels are not coincidental.
\end{proof}

Lemma~\ref{l:PQQP} in its second form is known, proved via minimal polynomial methods~\cite{KeithConrad} rather than via invariant subspaces.

\subsubsection*{Proof of Theorem \ref{t:HVAS}}
For $\A$ and $\S$ to commute, we have to show that they have the same eigenspaces. Let us begin with the kernel, the eigenspace of the $0$ eigenvalue. Recall that the invariant directions are $\{\Pi_k\rho_0\Pi_k\}_{k=1}^N$. Therefore, in the Bloch representation, we easily find a basis for the kernel of $\A+\delta \S$:
\[
  \vec{u}_{N^2-N+k}=(\Tr((\Pi_k\rho_0\Pi_k)\vec{\sigma}_n))_{n=1}^{N^2}, \quad k=1,\dotsc,N.
\]
This kernel basis does not depend on $\delta$ and is the common kernel of $\A$ and $\S$. We further have the freedom to orthonormalize this basis by Corollary~\ref{c:orthonormal}. Next, for generically nonvanishing eigenvalues, elementary manipulations show that
\[
  (-\imath\Ad_H +\delta \mathfrak{L}(V))(\Pi_k\rho_0\Pi_\ell)=(-\imath \omega_{k\ell}+\delta \gamma_{k\ell})(\Pi_k\rho_0\Pi_\ell),
\]
i.e., $(\Pi_k\rho_0\Pi_\ell)$ is an eigenvector of the right-hand side of Eq.~\eqref{e:QME1} associated with the eigenvalue $-\imath \omega_{k\ell}+\delta \gamma_{k\ell} \ne 0$. Thus, in the Bloch representation the eigenvectors associated with the nonvanishing eigenvalues of $\A+\delta \S$ are
\begin{equation}\label{e:eig}
  \vec{u}_{k \ne \ell}=(\Tr((\Pi_k\rho_0\Pi_\ell)\vec{\sigma}_n))_{n=1}^{N^2}, \quad 1 \leq k \ne \ell \leq N.
\end{equation}
The eigenvalues remain the same, as can be seen from the commutativity of the following diagram:
\[
  \begin{array}{ccc}
  \mathbf{Herm} & \stackrel{-\imath\Ad_H +\delta \mathfrak{L}(V)}{\longrightarrow} & \mathbf{Herm} \\
  \downarrow &                                              & \downarrow \\
  \mathbb{R}^{N^2} & \stackrel{\A+\delta \S}{\longrightarrow} & \mathbb{R}^{N^2}
  \end{array}
\]
together with the linearity of the Bloch representation $\downarrow$. Relabel the eigenvectors in Eq.~\eqref{e:eig} as $\{\vec{u}_n\}_{n=1}^{N^2-N}$. Lemma~\ref{l:orthog} shows this set is orthonormal, and together with the kernel defines a unitary matrix $(\vec{u}_1,\dotsc, \vec{u}_{N^2-N}, \vec{u}_{N^2-N+1}, \dotsc, \vec{u}_{N^2})$ that diagonalizes $\A+\delta \S$ for all $\delta$. $U$ does not depend on $\delta$ and \emph{simultaneously} diagonalizes $\A$ and $\S$: setting $\delta=0$ implies that $\U$ diagonalizes $\A$. $\U$ also diagonalizes $\frac{1}{\delta+1} \A+\frac{\delta}{1+\delta}\S$. Setting $\delta \uparrow \infty$ implies that $\U$ diagonalizes $\S$. Thus $\A$ and $\S$ have the same eigenvectors and $[\A,\S]=0$.\hfill$\blacksquare$

\bibliographystyle{IEEEtran}
\bibliography{biblio/edmond,biblio/physics,biblio/oxford}

\begin{IEEEbiography}[{\includegraphics[width=1in, height=1.25in,clip,keepaspectratio]{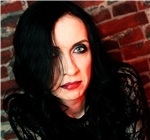}}]{Sophie Schirmer (Shermer)} is an Associate Professor in Physics at Swansea University, UK, and previously held positions as Advanced Research Fellow of the Engineering \& Physical Sciences Research Council at Cambridge University, Visiting Professor at Kuopio University, Finland, and positions at the Open University and University of Oregon. SS's research interests include nano-science at the quantum edge and quantum engineering, especially modeling, control and characterization of quantum systems.
\end{IEEEbiography}

\begin{IEEEbiography}[{\includegraphics[width=1in,height=1.25in,clip,keepaspectratio]{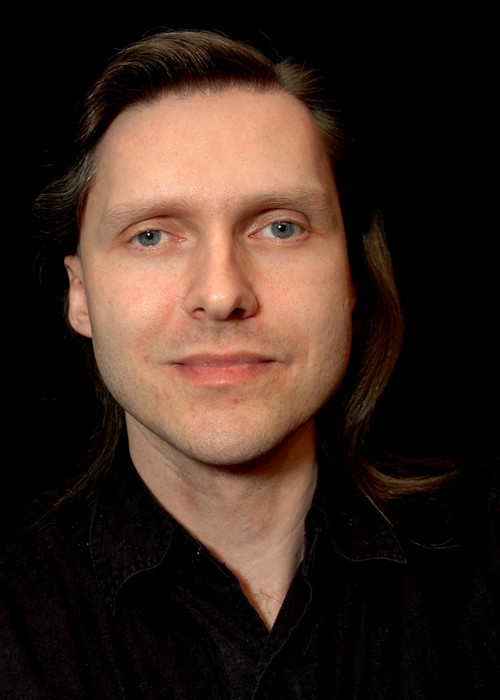}}]{Frank C.\ Langbein}
received his Mathematics degree from Stuttgart University, Germany in 1998 and a Ph.D. from Cardiff University, Wales, U.K. in 2003. He is currently a senior lecturer at the School of Computer Science and Informatics, Cardiff University, where he is a member of the visual computing group. His research interests include modeling, simulation, control and machine learning applied to quantum technologies, geometric modeling and healthcare. He is a member of the IEEE and the AMS.
\end{IEEEbiography}

\begin{IEEEbiography}[{\includegraphics[width=1in,height=1.25in,clip,keepaspectratio]{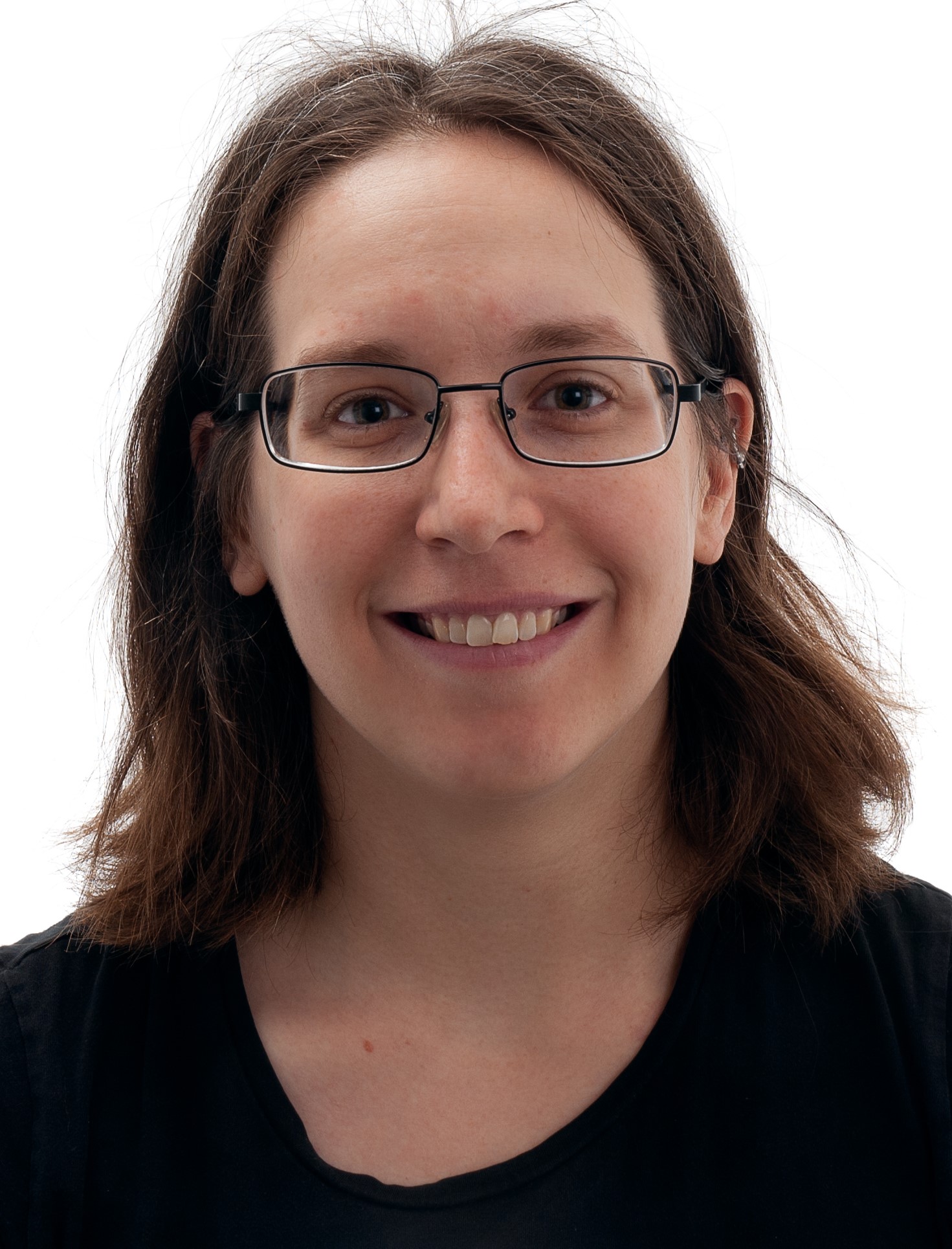}}]{Carrie A.\ Weidner}
received B.S. degrees in Engineering Physics and Applied Mathematics in 2010 and a Ph.D. in physics in 2018, all from the University of Colorado Boulder. After some time as a postdoctoral researcher, then assistant professor at Aarhus University, she is a lecturer at the University of Bristol Quantum Engineering Technology Laboratories. Her current research focuses on quantum control, sensing, and simulation, especially with ultracold atoms.
\end{IEEEbiography}

\begin{IEEEbiography}[{\includegraphics[width=1in,height=1.25in,clip,keepaspectratio]{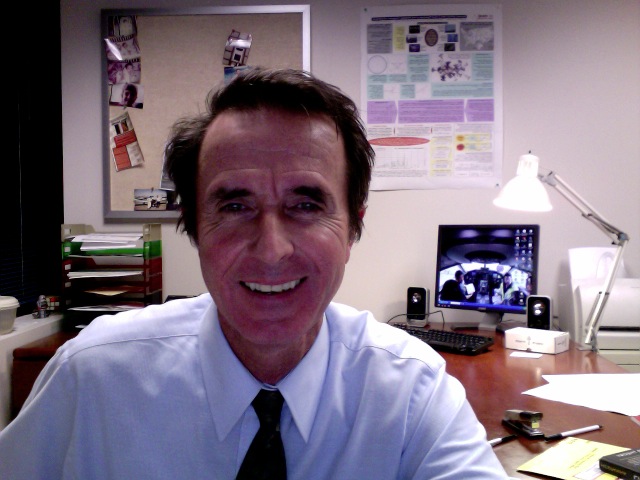}}]{Edmond A.\ Jonckheere}
received his Engineering degree from the University of Louvain, Belgium, in 1973, Dr.-Eng.\ in Aerospace Engineering from the Universit\'{e} Paul Sabatier, Toulouse, France, in 1975, and Ph.D. in Electrical Engineering from the University of Southern California in 1978. In 1973-1975, he was a Research Fellow of the European Space Agency; in 1979, he was with the Philips Research Laboratory, Brussels, Belgium, and in 1980 he joined the University of Southern California, where he is a Professor of Electrical Engineering and Mathematics and member of the Centers for Applied Mathematical Sciences and Quantum Information Science and Technology. He is a Life Fellow of the IEEE whose research interests include conventional vs quantum control, adiabatic quantum computations, wireless networking and the power grid.
\end{IEEEbiography}

\vfill

\end{document}